\documentclass{amsart}
\usepackage{latexsym}
\usepackage{amsmath,amsthm,amssymb, amscd,color}
\usepackage{etex, verbatim}
\usepackage[all]{xy}

\usepackage{tikz}
\usepackage{tikz-cd}
\usetikzlibrary{arrows,matrix,positioning}
\usetikzlibrary{fit}
\usetikzlibrary{patterns}
\input xy \xyoption{all}
\usepackage{blkarray} 

\usepackage{mathrsfs} 

\usepackage[utf8]{inputenc}
\usepackage[normalem]{ulem}

\theoremstyle{plain}
\newtheorem{theorem}{Theorem}[section]
\newtheorem{lemma}[theorem]{Lemma}
\newtheorem{corollary}[theorem]{Corollary}
\newtheorem{proposition}[theorem]{Proposition}
\numberwithin{equation}{section}

\theoremstyle{definition}
\newtheorem{definition}[theorem]{Definition}
\newtheorem{example}[theorem]{Example}

\newtheorem{remark}[theorem]{Remark}

\theoremstyle{remark}

\newcommand{\C}{\mathbb{C}}
\newcommand{\CP}{\mathbb{C}P}

\newcommand{\Q}{\mathbb{Q}}
\newcommand{\R}{\mathbb{R}}
\newcommand{\Z}{\mathbb{Z}}

\newcommand{\Sigmahat}{\widehat{\Sigma}}
\newcommand{\sigmahat}{\hat{\sigma}}

\newcommand{\im}{{\rm Im}\hspace{1pt}}

\makeatletter \def\blfootnote{\xdef\@thefnmark{}\@footnotetext}\makeatother

\numberwithin{equation}{section}

\newcommand{\PP}{\mbox{$\mathit P\hspace{-1pt}P$}}

\def\quasitoric{toric }

\def \({\left(}
\def \){\right)}
\def \<{\langle}
\def \>{\rangle}
\def \bar{\overline}
\def \deg{\mathrm{deg}}

\def \bb{\mathbb}

\def \mc{\mathcal}

\def \CP{{\bb{CP}}}

\def \PP{{\mathcal{PP}}} 
\def \SR{{\mathcal{SR}}} 

\def \ZZ{{\bb{Z}}}

\def \begineq{\begin{equation}}
\def \endeq{\end{equation}}

\numberwithin{equation}{section}

\begin{document}

\title[]{On the Integral Cohomology Ring of Toric Orbifolds and Singular Toric Varieties}

\author[A. Bahri]{Anthony Bahri}
\address{Department of Mathematics, Rider University}
\email{bahri@rider.edu}
%

\author[S. Sarkar]{Soumen Sarkar}
\address{Department of Mathematics, Indian Institute of Technology Madras}
\email{soumensarkar20@gmail.com}

\author[J. Song]{Jongbaek Song}
\address{Department of Mathematical Sciences, KAIST}
\email{jongbaek.song@gmail.com}
%



\dedicatory{}

\begin{abstract}
We examine the integral cohomology rings of certain families of
$2n$-dimensional orbifolds $X$ that are equipped with a well-behaved
action of the $n$-dimensional real torus. These orbifolds arise from two
distinct but closely related combinatorial sources, namely from 
characteristic pairs $(Q,\lambda)$,  
where $Q$ is a simple convex
$n$-polytope and $\lambda$ a labelling of its facets, and from
$n$-dimensional fans $\Sigma$. In the literature, they are referred 
as \quasitoric orbifolds and singular toric varieties respectively. 
Our first main result provides combinatorial conditions on 
$(Q,\lambda)$ or on $\Sigma$ which
ensure that the integral cohomology groups $H^{\ast}(X)$ of the associated 
orbifolds are concentrated in even degrees. Our second main result 
assumes these condition to be true, and expresses the graded ring 
$H^*(X)$ as a quotient of an algebra of polynomials that satisfy an 
integrality condition arising from the underlying combinatorial data. 
Also, we compute  several examples. 

\end{abstract}
\maketitle

\tableofcontents

\section{Introduction}
There are several advantages to studying topological spaces whose
integral cohomology groups $H^*(X)$ are torsion-free and concentrated
in even degrees; for example, their complex $K$-theory and complex
cobordism groups may be deduced immediately, because the appropriate
Atiyah-Hirzebruch spectral sequences collapse for dimensional
reasons. For convenience, we call such spaces \emph{even}, where
integral coefficients are understood unless otherwise stated. Our
fundamental aim is to identify certain families of even spaces within
the realms of toric topology, and to explain how their evenness leads
to a description of the Borel equivariant cohomology rings $H_T^*(X)$,
and thence to the multiplicative structure of $H^*(X)$.

Many even spaces arise from complex geometry, and have been of major
importance since the early 20th century. They range from complex
projective spaces and Grassmannian manifolds, to Thom spaces of
complex vector bundles over other even spaces. Examples of the latter
include stunted projective spaces, which play an influential and
enduring role in homotopy theory, and certain restricted families of
weighted projective spaces. In fact \emph{every} weighted projective
space is even, thanks to a beautiful and somewhat surprising result of
Kawasaki \cite{Ka}, whose calculations lie behind one of our main 
works in Section \ref{sec_van_odd_homology}. In the literature, 
weighted projective spaces have been viewed 
as singular toric varieties or as  \quasitoric orbifolds which we shall define 
in Section \ref{sec_toric_orb_and_orb_lens_sp}, and our results may 
be interpreted as an investigation of their generalizations within either 
context.

We begin in Section \ref{sec_ret_of_simple_poly} by introducing a 
sequence $\{B_k\}$ of polytopal complexes whose initial term is a 
simple polytope $Q$ and the final term is a vertex of $Q$. We define 
the sequence inductively by the rule stated as 
\eqref{def_weak_free_vertex} in Section \ref{sec_ret_of_simple_poly}, which is motivated by several 
spaces called \emph{invariant subspaces}, and \emph{orbifold 
lens spaces} sitting inside the given toric orbifold. 

In Section \ref{sec_toric_orb_and_orb_lens_sp},  we summarize the 
theory of \emph{toric orbifolds}
$X=X(Q,\lambda)$\footnote{In the literature, these orbifolds are 
sometimes called \emph{quasitoric} orbifolds}, as constructed from an 
$n$-dimensional simple convex polytope $Q$ and an 
\emph{$\mathcal{R}$-characteristic function} $\lambda$ from its facets to $\Z^n$. 
The combinatorial data $(Q, \lambda)$ is called 
an \emph{$\mathcal{R}$-characteristic pair} associated to the given toric orbifold. 
The notion of \emph{invariant subspaces} and 
\emph{orbifold lens spaces} follow from $(Q,\lambda)$, which we 
shall explain in the following subsections. 
Moreover, for each polytopal complex $B$ which appears in a 
retraction sequence, the $\mathcal{R}$-characteristic function $\lambda$ 
may be used to associate a finite group $G_{B}(v)$, see \eqref{eq_def_G_B(v)}, 
to certain vertices $v$ called free vertices in $B$, 
and to define the collection 
\begin{equation}\label{eq_gcd_collection}
\big\{ |G_{B}(v)| ~\big|  ~v ~\text{is a \emph{free vertex} in}~B \big\}.
\end{equation}

Interest in toric orbifolds was stimulated by Davis and Januszkiewicz
\cite{DJ}, who saw them as natural extensions to their own smooth toric
manifolds\footnote{They are renamed in \cite{BP} as 
\emph{quasitoric} manifolds.}. They proved that
\quasitoric manifolds are always even; 
however, the best comparable
statement for \quasitoric \emph{orbifolds} is due to Poddar and the second author 
\cite{PS} who showed that, in general, they are only even over the
rationals. 
%
We introduce our main result of the first part of this paper in 
Section \ref{sec_van_odd_homology} as follows. 
\begin{theorem}\label{thm_gcd_theorem}
Given any \quasitoric orbifold $X(Q,\lambda)$, assume that the \emph{gcd} of the
collection \eqref{eq_gcd_collection} is 1 for each $B$ which appears in 
a retraction sequence with $\dim B > 1$; then $X$ is even.
\end{theorem}
The proof employs a cofiber sequence 
involving \emph{orbifold lens spaces}, which
are generalization of \emph{lens complexes}, introduced by 
Kawasaki \cite{Ka}.  Furthermore, Theorem \ref{thm_gcd_theorem} automatically applies 
to weighted projective spaces.

In Section \ref{sec_cohom_ring_and_ppoly}, we restrict our emphasis to
projective toric orbifolds, which are realized as toric varieties whose details 
are admirably presented by Cox, Little and Schenck in their encyclopedic 
book \cite{CLS}. Every such variety
$X_\Sigma$ is encoded by a fan $\Sigma$ in $\R^n$, and admits a
canonical action by the $n$-dimensional real torus $T^n$. If $\Sigma$ is
\emph{smooth}, then the underlying geometry guarantees that $X_\Sigma$ is
always even. Moreover, it is true that the Borel equivariant cohomology
ring $H^*_T(X_\Sigma)$ is isomorphic to the Stanley--Reisner ring
$\mathcal{SR}[\Sigma]$, which is also concentrated in even degrees, and
$H^*(X_\Sigma)$ is its quotient by a linear ideal determined by \eqref{eq_linear_relations}. 
It is important to note that $\mathcal{SR}[\Sigma]$ is isomorphic to the ring
$\PP[\Sigma]$ of \emph{integral piecewise polynomials} on $\Sigma$,
for any \emph{smooth} fan.

For a particular class of singular examples, a comparable description
of the ring $H^*(X_\Sigma)$ was given in \cite{BFR}, as follows. If
$\varSigma$ is polytopal and $X_\Sigma$ is even, then $H^*(X_\Sigma)$
is the quotient of $\PP[\Sigma]$ by the ideal generated by all
\emph{global} polynomials. It is no longer possible to use the
Stanley--Reisner ring, which only agrees with $\PP[\Sigma]$ over the
rationals. In these circumstances, when $X_\Sigma$ is a toric variety over a 
polytopal fan, 
we have a major incentive to develop criteria which test whether or 
not it is even. There also remains the significant problem of 
presenting $\PP[\Sigma]$ by generators and relations, as exemplified 
by the calculation for the weighted projective space $\CP^3_{(1,2,3,4)}$ 
in \cite[\S4]{BFR}. So the aim of Section \ref{sec_cohom_ring_and_ppoly} 
is to find an alternative description for the ring of 
piecewise polynomials. 
It is accomplished by defining the \emph{weighted 
Stanley--Reisner ring} $w\mathcal{SR}[\Sigma]$, which turned out to be 
a subring of $\mathcal{SR}[\Sigma]$, consisting of polynomials
that satisfy an \emph{integrality condition}, see Definition \ref{def_int_cond}.  
The main result of the Section \ref{sec_cohom_ring_and_ppoly} 
combines Theorem \ref{thm_gcd_theorem} 
and Theorem \ref{thm_cohomology_ring}, as follows.
\begin{theorem}\label{thm_cohom_ring_without_Hodd=0}
Given any polytopal fan $\Sigma$ in $\R^n$, assume that 
the corresponding $\mc{R}$-characteristic pair 
$(Q, \lambda)$ satisfies the hypothesis of Theorem
\ref{thm_gcd_theorem};
then $X_\Sigma$ is even, and there exists an isomorphism
$$H^*(X_\varSigma)\;\cong\;w\mathcal{SR}[\Sigma] \big/ \mathcal{J}$$
of graded rings, where $\mathcal{J}$ is an ideal of linear 
relations determined by the generators of rays of~$\Sigma$.
\end{theorem} 
So our combinatorial condition on the fan allows us to give an
explicit description of the integral cohomology ring of $X_\Sigma$.

Several natural questions present themselves for future discussion.
For example, Sections \ref{sec_toric_orb_and_orb_lens_sp} and
\ref{sec_cohom_ring_and_ppoly} may be linked more closely by
establishing a common framework for \quasitoric 
orbifolds and toric varieties over non-smooth polytopal fans. 
The theory of multifans is an obvious candidate, but we
have been unable to identify an associated ring of piecewise
polynomials with sufficient clarity. However, the third author with Darby and Kuroki \cite{DKS}
has recently proposed a definition of piecewise polynomials on an 
orbifold torus graph, which does allow those two objects to 
be dealt with simultaneously.

In view of our opening remarks, another reasonable challenge is to
extend our study to the complex $K$-theory and complex cobordism of
\quasitoric orbifolds. This program was suggested by work
of Harada, Henriques and Holm \cite{HHH}, and begun in \cite{HHRW} by
the adoption of a categorical approach to piecewise structures; but
overall progress has been limited to a small subfamily of weighted
projective spaces, and much further work is required. 
However, some progress have done by the second author and Uma \cite{SU}.  

\textbf{Acknowledgments.}
We extend our sincere gratitude to Nigel Ray. Most properly, his name belongs
among those of the authors. The significance of his mathematical contribution 
to the results is most certainly not reflected accurately by the omission of his 
name. Among his many other contributions to the paper is the concise and 
elegant introduction. We thank Mikiya Masuda and Haozhi Zeng for pointing out 
a gap in the proof of Theorem 2.12 of the previous version, and Li Cai for helpful
comments about retraction sequences. 
We are also grateful for the hospitality of University of Calgary in July, 2015.

This work has been supported in part by Simons Foundation Grants 210368 and 426160. 
The first author acknowledges the award of research leave from Rider University and is grateful also for
the hospitality of the Princeton University Mathematics Department during the spring of 2015.
The second author would like to thank Pacific Institute for 
the Mathematical Sciences and University of Regina for financial support.
Lastly, the third author was supported by 
Basic Science Research Program through the National Research Foundation of Korea
(NRF) funded by the Korea government (MSIP) (No. NRF-2016R1A2B4010823).
He also would like to thank his advisor 
Prof. Dong Youp Suh for his encouragements and support throughout the project.

\section{A retraction of simple polytopes}
\label{sec_ret_of_simple_poly}

In this section, we introduce a natural way of retracting a 
simple polytope $Q$ to a point, which we call a 
\emph{retraction sequence}. For each polytope, there are finitely 
many such retractions, enabling us to develop a sufficient 
condition for torsion freeness in the homology of toric 
orbifolds in the following section. The operation itself is 
motivated by several spaces which arise in a toric orbifold
by decomposing the orbit space. We shall explain this
topological interpretation in Section \ref{sec_toric_orb_and_orb_lens_sp}. 
This section is devoted to give the combinatorial definition and properties of 
retraction sequences. 
We begin by introducing the definition of a polytopal complex. 

\begin{definition}\cite[Definition 5.1]{Zie}
A \emph{polytopal complex} $\mathcal{C}$ is a finite collection of 
polytopes in $\R^n$ satisfying:
\begin{enumerate}
\item if $E$ is a face of $F$ and $F \in \mathcal{C}$ then 
$E \in \mathcal{C}$,
\item if $E, F \in \mathcal{C}$ then $E \cap F$ is a face of 
both $E$ and $F$.
\end{enumerate}
Let $|\mathcal{C}| = \bigcup_{F \in \mathcal{C}} F$ be the
 underlying set of $\mathcal{C}$. 
\end{definition}
The elements of $\mathcal{C}$ are called faces and the zero 
dimensional faces of $\mathcal{C}$ are called vertices. 
We denote the set of vertices of $\mathcal{C}$ by 
$V(|\mathcal{C}|)$.
The dimension of $\mathcal{C}$ or $|\mathcal{C}|$ is the 
maximum of the dimension of its faces. Given a simple 
polytope $Q$, let $\mathcal{C}(Q)$ be the collection 
of all faces of $Q$ and $\mathscr{F}(Q)$ the collection of all 
facets of $Q$. Then, $\mathcal{C}(Q)$ is a polytopal 
complex and $|\mathcal{C}(Q)|$ is homeomorphic to 
$Q$ as manifolds with corners. 
Through out this paper, we always assume that 
$\ell:=|V(Q)|$, the number of vertices of $Q$,  
$m:=|\mathscr{F}(Q)|$, the number of facets of $Q$ and $n:=\dim Q$. 

Now, given an $n$-dimensional simple polytope $Q$, we 
construct a sequence of triples 
$\{(B_k, E_k, b_k)\}_{k=1}^\ell$, which we call a 
{\it retraction sequence} of $Q$. First, we define 
$B_1 = Q=E_1$ and $b_1 \in V(B_1)$. The second term 
$(B_2, E_2, b_2)$ is defined as follows. Consider a 
subcollection 
$$\mathcal{C}_2= \{E \in\mathcal{C}(Q) \mid  b_1\notin V(E)\}$$ 
of $\mathcal{C}(Q)$.  
Then, $\mathcal{C}_2$ is an $(n-1)$-dimensional polytopal 
complex. We define $B_2$ by the underlying set 
$|\mathcal{C}_2|$ of $\mathcal{C}_2$. We choose vertex $b_2$ 
of $B_2$ such that $b_2$ has a neighborhood 
diffeomorphic to $\R^N_{\geq 0}$ as manifold with corners, for 
some $1\leq N \leq \dim B_2$  and let  $E_2$ be 
the  unique $N$-dimensional  face of $B_2$ containing $b_2$.
 Notice that, in this case,  $N=n-1$ and we have $n$ many 
 different choices of $b_2$ because $Q$ is an $n$-dimensional simple polytope. 

Next, we construct the sequence inductively. Given  
$(B_k, E_k, b_k)$,  the next term 
$(B_{k+1}, E_{k+1}, b_{k+1})$ is defined as follows. First we 
consider a polytopal complex 
$$\mathcal{C}_{k+1} =
\{ E\in \mathcal{C}_{k} \mid b_{k} \notin V(E)\}.$$
Then, $B_{k+1}$ is defined by its underlying set 
$|\mathcal{C}_{k+1}|$. 
We choose a vertex $b_{k+1}$ in $V(B_{k+1})$ satisfying the 
following condition:
\begin{align*}\tag{\small{$\blacklozenge$}}\label{def_weak_free_vertex}
&``b_{k+1} \text{ has a neighborhood 
homeomorphic to } \R^{N}_{\geq 0} \text{ as manifold with corners}, \\
&\text{for some } N\in \{1, \dots, \dim B_{k+1}\} "
\end{align*}
and $E_{k+1}$ defined to be a unique face of $B_{k+1}$ containing 
$b_{k+1}$ with $\dim E_{k+1}=N$. 
\begin{definition}\label{def_free_vertex}
We call a vertex $v$ in $B_k$ a \emph{free vertex}, if it has a neighborhood in $B_k$
that is diffeomorphic to $\R^N_{\geq 0}$ as manifold with corners, for some 
$N\in \{1, \dots, \dim B_k\}$. We denote the set of free vertices in $B_k$
by $FV(B_k)$.
\end{definition}

 The proof of Proposition \ref{prop_existence_of_ret} below guarantees the 
existence of free vertices at each step. Finally, the sequence stops if 
the sequence reaches a vertex, i.e., 
$B_\ell=E_\ell=b_\ell \in V(Q)$. Essentially, we can think of
a retraction sequence as an iterated choice of free
vertices at each step.  Figure \ref{fig_ret_of_vc_of_cube} 
shows an example of retraction sequence for the vertex cut of a 
cube, where the colored face of each $B_k$ indicates $E_k$ for 
$k=1, \dots, 10$. 

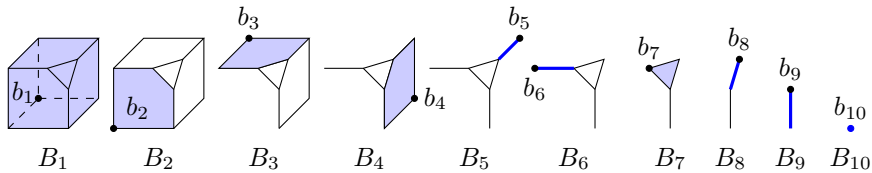
\begin{figure}
        \begin{tikzpicture}[scale=0.4]
        		\draw[fill=blue!20, blue!20] 
			(0,0)--(2,0)--(3,1)--(3,3)--(1,3)--(0,2)--cycle;
		\draw(0,0)--(2,0)--(2,1.3)--(1.3,2)--(0,2)--cycle;
		\draw (2,1.3)--(2.3,2.3)--(1.3,2);
		\draw (0,2)--(1,3)--(3,3)--(2.3,2.3);
		\draw[dashed] (1,3)--(1,1)--(0,0);
		\draw (3,3)--(3,1)--(2,0);
		\draw[dashed]  (1,1)--(3,1);
		\node[left] at (1.2,1.2) {$b_1$};
		\draw[fill] (1,1) circle [radius=0.1];
		\node at (1.5,-1) {$B_1$};
		
		\begin{scope}[xshift=35mm]
		\draw[fill=blue!20]
		 (0,0)--(2,0)--(2,1.3)--(1.3,2)--(0,2)--cycle;
		\draw (2,1.3)--(2.3,2.3)--(1.3,2);
		\draw (0,2)--(1,3)--(3,3)--(2.3,2.3);
		\draw (3,3)--(3,1)--(2,0);
		\node[above] at (0.8,0) {$b_2$};
		\draw[fill] (0,0) circle [radius=0.1];
		\node at (1.5,-1) {$B_2$};
		
		\begin{scope}[xshift=35mm]
    		\draw (2,0)--(2,1.3)--(1.3,2)--(0,2);
		\draw (2,1.3)--(2.3,2.3)--(1.3,2);
		\draw[fill=blue!20]
		 (0,2)--(1,3)--(3,3)--(2.3,2.3)--(1.3,2)--cycle;
		\draw (3,3)--(3,1)--(2,0);
		\node[above] at (1,3) {$b_3$};
		\draw[fill] (1,3) circle [radius=0.1];
		\node at (1.5,-1) {$B_3$};
		
		\begin{scope}[xshift=35mm]
		\draw (2,0)--(2,1.3)--(1.3,2)--(0,2);
		\draw (2,1.3)--(2.3,2.3)--(1.3,2);
		\draw (3,3)--(2.3,2.3);
		\draw[fill=blue!20] 
		(3,3)--(3,1)--(2,0)--(2,1.3)--(2.3,2.3)--cycle;
		\node[right] at (3,1) {$b_4$};
		\draw[fill] (3,1) circle [radius=0.1];
		\node at (1.5,-1) {$B_4$};
		
		\begin{scope}[xshift=35mm]
		\draw (2,0)--(2,1.3)--(1.3,2)--(0,2);
		\draw (2,1.3)--(2.3,2.3)--(1.3,2);
		\draw[blue, very thick] (3,3)--(2.3,2.3);
		\node[above] at (3,3) {$b_5$};
		\draw[fill] (3,3) circle [radius=0.1];
		\node at (1.5,-1) {$B_5$};
		
		\begin{scope}[xshift=35mm]
		\draw (2,0)--(2,1.3)--(1.3,2)--(0,2);
		\draw (2,1.3)--(2.3,2.3)--(1.3,2);
		\node[below] at (0,2) {$b_6$};
		\draw[blue, very thick] (0,2)--(1.3,2);
		\draw[fill] (0,2) circle [radius=0.1];
		\node at (1.3,-1) {$B_6$};
		
		\begin{scope}[xshift=25mm]
		\draw (2,1.3)--(2,0);
		\draw[fill=blue!20]  
		(2,1.3)--(2.3,2.3)--(1.3,2)--cycle;
		\node[above] at (1.3,2) {$b_7$};
		\draw[fill] (1.3,2) circle [radius=0.1];
		\node at (2,-1) {$B_7$};

		\begin{scope}[xshift=20mm]
		\draw (2,0)--(2,1.3)--(2.3,2.3);
		\draw[blue, very thick] (2.3,2.3)--(2,1.3);
		\node[above] at (2.3,2.3) {$b_8$};
		\draw[fill] (2.3,2.3) circle [radius=0.1];
		\node at (2,-1) {$B_8$};
		
		\begin{scope}[xshift=20mm]
		\draw[blue, very thick] (2,0)--(2,1.3);
		\node[above] at (2,1.3) {$b_9$};
		\draw[fill] (2,1.3) circle [radius=0.1];
		\node at (2,-1) {$B_9$};
		
		\begin{scope}[xshift=20mm]
		\node[above] at (2,0) {$b_{10}$};
		\draw[fill=blue, blue] (2,0) circle [radius=0.1];
		\node at (2,-1) {$B_{10}$};
		
		\end{scope}
		\end{scope}
		\end{scope}		
		\end{scope}
		\end{scope}
		\end{scope}
		\end{scope}
		\end{scope}				
		\end{scope}
		
        \end{tikzpicture}
        \caption{A retraction sequence of a vertex cut of the cube.}
        \label{fig_ret_of_vc_of_cube}
\end{figure}

\begin{proposition}\label{prop_existence_of_ret}
Every simple polytope has at least one retraction sequence.
\end{proposition}
\begin{proof}
We begin by following the argument of \cite[Proposition 3.1]{DJ}. 
First, we realize $Q$ as a convex polytope in $\R^n$ and 
choose a vector $u\in \R^n$ such that 
$$\langle u, v\rangle \neq \langle u, w\rangle,~
\text{whenever }v \neq w \in V(Q) \subset \R^n,$$
with respect to the Euclidean inner product $\langle~,~\rangle$.
Let $e:=e(vw)$ be the oriented edge with the initial vertex 
$i(e)=v$ and the terminal vertex $t(e)=w$. Here, the direction of 
$e(vw)$ is given by the following rule:
$$i(e)=v~ \text{and}~t(e)=w,~~\text{if and only if}~ 
\langle u,v \rangle < \langle u, w\rangle,$$
which makes the one skeleton of $Q$ into a directed graph. 
	
Let $ind(v)$ be the number of inward edges at $v$ and we call 
$ind(v)$ the index of $v$ (with respect to the choice of generic 
vector $u$).  Then, for each face $E \subset Q$, there exists the 
unique vertex $v$ of $E$ having the maximal index among the 
vertices in $E$. Moreover, $E$ is locally diffeomorphic to 
$\R^{ind(v)}_{\geq 0}$ around $v$. 
Conversely, given  a vertex $v\in V(Q)$, there exists a unique face 
$E_v$ such that $\dim E_v= ind(v)$.	

Let $\{b_k\}_{k=1}^\ell$ be a sequence  of vertices in $Q$ 
determined by 
$$\langle u, b_1 \rangle > \langle u, b_2\rangle > \dots
>  \langle u, b_\ell \rangle.$$
Notice that $ind(b_1)=n=\dim Q$, and $ind(b_\ell)=0$. 
Now, we claim that the following sequence 
$$\left\{ \left(B_k:= \bigcup_{j\geq k} E_{b_j}, E_{b_k}, b_k 
\right) \right\}_{k=1, \dots, \ell},$$
where $E_{b_k}$ is a unique face containing $b_k$ with 
$\dim E_{b_k} = ind(b_k)$, is a retraction sequence 
of $Q$. Indeed, for each $k\in \{1, \dots, \ell-1\}$, we have
$\langle u, b_k \rangle > \langle u, v \rangle$, for all 
$v\in V(B_k)\setminus \{b_k\}$. Hence, there are no outgoing 
edges from $b_k$ in $B_k$, which implies that $b_k$ has a
neighborhood in $E_{b_k} \subseteq B_k$ homeomorphic to 
$\R^{ind(b_k)}_{\geq 0}$ as manifold with corners. 
\end{proof}

We denote by  $\mathfrak{R}(Q)$ the set of all 
retraction sequences of $Q$ and by $\mathfrak{B}(Q)$ the set of 
all possible $B_i$'s which appear in $\mathfrak{R}(Q)$. 
Evidently, both $\mathfrak{R}(Q)$ and $\mathfrak{B}(Q)$ are finite
 sets, because we have finitely many choice of free vertices 
at each step. 

\begin{remark}
The retraction sequence has a strong relation with
\emph{shelling} of a simplicial complex. We are preparing an 
independent article \cite{BSS-shelling} about the exact 
correspondence and some other interesting properties. 
\end{remark}


\section{Toric orbifolds and orbifold lens spaces}
\label{sec_toric_orb_and_orb_lens_sp}
In this section we recall the \emph{characteristic pairs}
$(Q,\lambda)$ of \cite{DJ} and \cite{PS}, and explain the way in 
which they are used to construct \emph{\quasitoric orbifolds} 
$X=X(Q,\lambda)$. 
If $\lambda$ obeys Davis and Januszkiewicz's condition $(*)$
(see \cite[page 423]{DJ}), then $X$ is smooth and even; so one of 
the main goals of this paper is to establish
Theorem \ref{thm_gcd_theorem}, which focuses on 
singular cases,  and states a sufficient condition for the orbifold 
$X$ to be even. In this section, to complete the proof of Theorem 
\ref{thm_gcd_theorem}, we commandeer two additional types of 
spaces, namely the \emph{invariant subspaces} of $X$ which 
arise as the preimage of faces via the orbit map, 
and the \emph{orbifold lens spaces} that arise as quotients of 
odd dimensional spheres by the actions of certain finite groups 
associated to $\lambda$. 

%


\subsection{Toric orbifolds}\label{subsec_toric_orb}
In this subsection, we discuss a combinatorial definition of toric orbifolds. 
Let $Q$ be an $n$-dimensional simple convex polytope in $\R^n$ 
and $\mathscr{F}(Q)=\{F_1, \dots, F_m\}$ the codimension one 
faces of $Q$ which are called \emph{facets}. 
\begin{definition}\label{def_char_function}
A function $\lambda \colon \mathscr{F}(Q) \to \Z^n$ is called a
\emph{rational characteristic function} 
(or \emph{$\mathcal{R}$-characteristic function}) 
for $Q$ if it satisfies the following 
condition: 
\begin{equation}\label{eq_non-sing_cond.}
	\big\{ \lambda(F_{i_1}), \dots, \lambda(F_{i_k}) \big\} 
	\text{ is linearly independent, whenever } 
	\bigcap_{j=1}^k F_{i_j} \neq \varnothing.
\end{equation}
We denote $\lambda_i=\lambda(F_i)$ and call it an 
\emph{$\mathcal{R}$-characteristic vector} 
assigned to the facet $F_i$. 
The pair $(Q, \lambda)$ is called an $\mathcal{R}$-characteristic pair.
\end{definition}

\begin{remark}\label{rmk_char_matrix}
\begin{enumerate}
\item In the literature about toric manifolds,
the pair $(Q, \lambda)$ satisfying the condition $(\ast)$ in \cite[p. 423]{DJ}
is called a \emph{characteristic pair}.  
\item For convenience, we usually express an 
$\mathcal{R}$-characteristic function 
$\lambda$ as an $(n\times m)$-matrix $\Lambda$ by listing 
$\lambda_i$'s as column vectors. We call $\Lambda$ an 
\textit{$\mathcal{R}$-characteristic matrix} associated to $\lambda$. 	
\item It is easy to check that it suffices to satisfy the linearly 
independence at each vertex which is an intersection of $n$ facets. 
\end{enumerate}
\end{remark}

One canonical example of such function can be given by a 
\emph{simple lattice polytope} which is a convex hull of finitely many points in the integer 
lattice $\Z^n\subset \R^n$ and simple. Namely, we can naturally assign an 
$\mathcal{R}$-characteristic vector the 
primitive normal vector on each facet of a simple lattice polytope. In Section 
\ref{sec_cohom_ring_and_ppoly}, we shall see this again as 
primitive vectors 
of $1$-dimensional cones in a normal fan associated to a simple 
lattice polytope. 

For $x\in Q$, we denote by $E(x)$ the face of $Q$ which contains
$x$ in its interior. If $E(x)$ is a face of codimension $k$, then it is 
a unique intersection of $k$ facets 
$F_{i_1},\dots, F_{i_k}$. We also denote by $T_{E(x)}$ 
the subtorus of standard $n$-dimensional torus $T^n$, determined 
by $\lambda_{i_1}, \dots, \lambda_{i_k}$. To be more precise, we may
regard the target space $\Z^n$ of $\lambda$ as the $\Z$-submodule
of the Lie algebra of $T^n$, and $T_{E(x)}$ is the torus generated 
by the exponential image of the lines determined by 
the $\mathcal{R}$-characteristic 
vectors $\lambda_{i_1}, \dots, \lambda_{i_k}$. 

Now, we define an equivalence relation $\sim$ on the product 
$T^n \times Q$ by
\begin{equation}\label{eq_equiv_rel_of_toric_orb}
(t, x)\sim_\lambda (s,y) \text{ if and only if } x=y \text{ and }
t^{-1}s \in T_{E(x)}.
\end{equation}
The quotient space 
$$ X(Q, \lambda)=(T^n \times Q)/\sim_\lambda$$
has an orbifold structure with a natural $T^n$-action induced by 
the group operation, see Section 2 in \cite{PS}. Clearly, the orbit 
space of $T^n$-action on $X(Q, \lambda)$ is $Q$. Let 
\begin{equation}\label{eq_orbit_map}
\pi: X(Q, \lambda) \to Q,\quad   \pi([t,x]_{\sim_\lambda}) = x
\end{equation}
be the orbit map, where $[t,x]_{\sim_\lambda}$ is the equivalence class of $(t,x)$
with respect to $\sim_\lambda$. The space $X(Q, \lambda)$ is 
called the \quasitoric orbifold associated to the combinatorial pair 
$(Q, \lambda)$. 

In analyzing the orbifold structure of $X(Q, \lambda)$, Poddar and
 Sarkar, \cite[Subsection $2.2$]{PS}, gave an axiomatic definition 
 of \quasitoric orbifolds, which generalizes the axiomatic definition 
 of \quasitoric manifolds of \cite{DJ}.


\subsection{Invariant subspaces}\label{subsec_inv_subsp}
In this subsection, we study the $\mc{R}$-characteristic pair of 
some invariant subspaces of $X(Q, \lambda)$. 
Let $E=F_{i_1}\cap  \dots \cap  F_{i_k}$ be a face of $Q$, 
where $F_{i_1}, \dots, F_{i_k}$ are facets. We can define a 
natural projection 
\begin{equation}\label{eq_def_of_rho_E}
\rho_E \colon \Z^n \to \Z^n / (({\rm span}\{\lambda_{i_1}, \dots, 
\lambda_{i_k}\}\otimes_\Z \R) \cap \Z^n),
\end{equation}
where the target space is isomorphic to $\Z^{n-k}$, because 
$({\rm span}\{\lambda_{i_1}, \dots, \lambda_{i_k}\}\otimes_\Z \R) 
\cap \Z^n$ is a rank $k$ direct summand of $\Z^n$.
Notice that the rank of the target space of $\rho_E$ is same 
as the dimension of $E$. We consider $E$ as an independent simple 
polytope, and denote the set of facets of $E$ by 
$$\mathscr{F}(E) = \{ E\cap F_j \mid F_j \in \mathscr{F}(Q)~ 
\text{and}~ j\neq i_1, \dots, i_k, ~\text{and}~ E\cap F_j\neq 
\varnothing\}.$$

Now, the map $\rho_E$ together with $\lambda$ yields an  
$\mathcal{R}$-characteristic function 
\begin{equation}\label{eq_def_of_induced_char_ftn}
\lambda_E \colon \mathscr{F}(E) \to \Z^{n-k},
\end{equation}
on $E$ defined for $\lambda_E(E\cap F_j)$ to be 
the primitive vector of $(\rho_E\circ \lambda)(F_j)$.  
Indeed, the condition \eqref{eq_non-sing_cond.} naturally follows 
from $\lambda$. 

Hence, we get an $\mathcal{R}$-characteristic pair 
$(E, \lambda_E)$ from $(Q, 
\lambda)$, which yields another toric orbifold 
$$X(E, \lambda_E):= (T^{n-k} \times E)/\sim_{\lambda_E},$$
where the equivalence relation $\sim_{\lambda_E}$ defined in 
a manner similar to \eqref{eq_equiv_rel_of_toric_orb}.

\begin{proposition}\cite[Section 2.3]{PS}\label{prop_inv_subsp_is_toric_orb}
Let $\pi \colon X(Q, \lambda) \to Q$ and $(E, \lambda_E)$ be as 
above. Then, $\pi^{-1}(E)$ is a $T^n$-invariant suborbifold. 
Moreover, it is a toric orbifold homeomorphic to $X(E, \lambda_E)$ as 
a topological space. 
\end{proposition}

The second assertion of the above proposition follows from the 
fact that the circle subgroups determined by $\lambda_E(E\cap F_j)$ 
and $(\rho_E \circ \lambda)(F_j)$, respectively, are identical.  
We also remark that the torus $T^{n-k}$ acting on $X(E, \lambda_E)$
can be identified with the image of the map 
\begin{equation}\label{eq_induced_torus_map}
\bar\rho_E \colon T^n \to T^{n-k},
\end{equation}
which is induced from the map $\rho_E$.

\begin{example}\label{ex_explicit_computation_of_local_group}
Suppose we have an $\mathcal{R}$-characteristic pair 
$(Q, \lambda)$ as
described in Figure \ref{fig_fake_weighted_proj_sp_as_subsp}.
Notice that $Q$ is a $3$-dimensional polytope with 5 facets, say 
$\mathscr{F}(Q)=\{F_1, \dots, F_5\}$. 
Here, we assume that the target space $\Z^3$ of $\lambda$ is 
generated by the standard basis $\{e_1, e_2, e_3\}$.
We choose $E$ to be the facet $F_5$. So, $k=1$ and $n-k=2$.  
Then, the projection 
$$\rho_E \colon \Z^3 \to \Z^3 / \langle e_3 \rangle = 
\left<e_1, e_2, e_3 \right> / \left<e_3\right> \cong \Z^2$$
is onto the first two coordinates. The facets of $E$ are 
$F_2 \cap E, ~F_3\cap E  \text{ and } F_4 \cap E.$
Hence, the map 
$$\lambda_{E} \colon \{F_2 \cap E, ~F_3\cap E,~ F_4 \cap E\}
 \longrightarrow \mathbb{Z}^{2}$$
is defined by 
\begin{align*}
\lambda_{E}(F_2 \cap E) &= \mathcal{\rho}_{E}
(\lambda(F_{2}))=  (2,-1)= 2e_{1}-e_{2},\\
\lambda_{E}(F_3\cap E) &= \mathcal{\rho}_{E}(\lambda(F_{3}))=
(-1,-1)= -e_{1}-e_{2},\\
\lambda_{E}(F_4 \cap E) &= \mathcal{\rho}_{E}(\lambda(F_{4}))=
(-1,2)= -e_{1}+2e_{2}.
\end{align*}
The orbifold corresponding to $(E, \lambda_E)$ is known
to be a fake weighted projective space with weight $(1,1,1)$. We 
refer \cite{Bu} and \cite{Kasp} for the details of fake weighted 
projective space. 
\end{example}

\begin{figure}
        \begin{tikzpicture}[scale=0.5, yscale=0.8]
        	\draw[fill=yellow, yellow] (0,0)--(5,0)--(3,-1)--cycle;
        	\draw (5,0)--(5,4)--(0,4)--(0,0);
	\draw (0,4)--(3,3)--(3,-1)--(0,0);
	\draw (5,4)--(3,3);
	\draw (5,0)--(3,-1);
	\draw[dashed] (0,0)--(5,0);
	\node[left] at (0,4) {$v_{124}$};
	\node[right] at (5,4) {$v_{134}$};
	\node[below] at (2.3,3) {$v_{123}$};
	\node[left] at (0,0) {$v_{245}$};
	\node[right] at (5,0) {$v_{345}$};
	\node[below] at (3,-1) {$v_{235}$};	
	\node at (1.5,1.5) {\small$F_2$};
	\node at (4,1.5) {\small$F_3$};
	\node at (3,3.6) {\small$F_1$};
	\node[below] at (6,3) {\small$F_4$};
	\draw[thick, dotted, ->] (6,3) to [out=120,in=60] (3.5,2);
	\node[above] at (7,0) {\small$F_5$};
	\draw[thick, dotted, ->] (7,0) to [out=270,in=330] (3.5,-0.5);
	\node at (13,2) 
	{\small{$ \Lambda=\begin{blockarray}{ccccc}
\lambda_1 &\lambda_2&\lambda_3&\lambda_4&\lambda_5 \\
\begin{block}{[ccccc]}
2 &2& -1& -1& 0 \\		
3& -1& -1& 2& 0 \\
5& 0 &-2 &2 &1    \\
\end{block}
\end{blockarray}  $}};	
        \end{tikzpicture}
        \caption{}
        \label{fig_fake_weighted_proj_sp_as_subsp}
\end{figure}

\subsection{Orbifold lens spaces}\label{subsec_orb_lens_sp}
Here, we introduce a generalization of \emph{lens complexes}
and study their homology groups. 
Let $\Delta^{n-1}$ be the $(n-1)$-dimensional simplex and 
$\mathscr{F}(\Delta^{n-1})=\{F_1, \dots, F_n\}$ the facets of 
$\Delta^{n-1}$. 
We begin by introducing the following definition. 
\begin{definition}\label{def_rat_char_ftn}
A function $\xi \colon \mathscr{F}(\Delta^{n-1}) \to \Z^n$ is called 
an \emph{$\mathcal{L}$-characteristic function} on $\Delta^{n-1}$ if 
$\{ \xi(F_1), \dots, \xi(F_n)\}$ is linearly independent. We set
$\xi_i := \xi(F_i)$ for $i=1, \dots, n$.  
\end{definition}
Now, we define an equivalence relation $\sim_\xi$ on $T^{n}
\times \Delta^{n-1}$ as follows:
\begin{equation}\label{eq_equiv_rel_orb_lens_sp}
(t,x) \sim_{\xi} (s,y) \text{ if and only if }
x=y \text{ and  } t^{-1}s \in T_{F(x)},
\end{equation}
where $F(x)$ is the face containing $x$ in its interior and 
$T_{F(x)}$ denotes the subtorus of $T^n$ determined by 
$\xi_{i_1},\dots , \xi_{i_k}$, 
if $F(x)=F_{i_1}\cap\dots \cap F_{i_k}$. 
The pair $(\Delta^{n-1}, \xi)$ together with the equivalence 
relation $\sim_\xi$ yields the following quotient space:
$$L(\Delta^{n-1}, \xi):= T^n \times \Delta^{n-1} / \sim_\xi,$$
which we call the \emph{orbifold lens space} associated to 
$(\Delta^{n-1}, \xi)$. 

\begin{proposition}
\label{prop_orb_lens_sp=sphere_finite_gp}
The orbifold lens space $L(\Delta^{n-1}, \xi)$ is homeomorphic to 
the quotient space of the $(2n-1)$-dimensional  sphere $S^{2n-1}$ 
by the action of a finite group 
$G_\xi:=\Z^n/ \text{span} \{\xi_1, \dots, \xi_n\}$. 
\end{proposition}
\begin{proof}
The proof is essentially same as the proof of \cite[Proposition 2.3]{SS}. 
\end{proof}

\begin{remark}\label{rmk_orb_lens_sp_finite_gp}
\begin{enumerate}
\item In \cite{SS}, the function $\xi$ is called  a 
\emph{hyper-characteristic function} if the submodule generated 
by $\{  \xi(F_{i_1}), \ldots, \xi(F_{i_k})\}$ is a direct summand of 
$\ZZ^{n+1}$ of rank $k$, whenever $F_{i_1} \cap \dots \cap 
F_{i_k}$ is nonempty. 
In particular, if $\{  \xi(F_{i_1}), \ldots, \xi(F_{i_n})\}$
is a linearly independent set, then it becomes an 
$\mathcal{L}$-characteristic function. 
\item The action of $G_\xi$ is induced from the standard $T^n$-
action on $S^{2n-1}\subset \C^n$. 
\item The order $|G_\xi|$ of $G_\xi$ 
is exactly same as the determinant of the $n\times n$ matrix $
\left[ \begin{array}{c|c|c}
\xi_1 & \cdots & \xi_n\end{array} \right]$. 
\end{enumerate}
\end{remark}

Proposition 
\ref{prop_orb_lens_sp=sphere_finite_gp} leads us the following 
lemma. 
\begin{lemma}\label{lem_homology_orb_lens_sp}
Let $p_1, \dots,  p_r$ be the prime factors of $|G_\xi|$. Then, 
$$H_j(L(\Delta^{n-1}, \xi))= \begin{cases} \Z & \text{ if } j=0, 2n-1\\
G_j &\text{ if } 1\leq j \leq 2n-2 \end{cases}, $$ 
where $G_j =(\Z/p_1^{a_{j_1}}\Z) \oplus \dots \oplus 
(\Z/p_r^{a_{j_r}}\Z)$ for some non-negative integers 
$a_1, \dots, a_r$. 
\end{lemma}
\begin{proof}	
We see $H_0(L(\Delta^{n-1}, \xi))\cong \ZZ$ trivially. The 
isomorphism  $H_{2n-1}(L(\Delta^n, \xi))\cong \Z$ follows 
because $G_\xi$ action on $S^{2n-1}$ is induced from the 
standard action of $T^{n}$ on $S^{2n-1} \subset \C^{n}$, 
which is orientation preserving. 
For $j\in \{1, \dots, 2n-2\}$, recall the following isomorphism 
which can be obtained from the classical result for an action of 
a finite group $G$ on a locally compact Hausdorff space $X$:
\begin{equation}\label{eq_transfer_isom}
H^\ast(X/G;\mathbf{k}) \cong H^\ast(X;\mathbf{k})^G,
\end{equation}
where $\mathbf{k}$ is a field of characteristic zero or prime to 
$|G|$, see \cite[III.2]{Bor}. 
	
We apply the isomorphism \eqref{eq_transfer_isom} to the 
orbifold lens space $L(\Delta^{n-1}, \xi)\cong S^{2n-1}/G_\xi$. 
Since $H^{j}(S^{2n-1};\mathbf{k})^{G_\xi}=0$ for $j=1, \dots 2n-2$, 
the claim is proved by the universal coefficient theorem. 
\end{proof}

Toric orbifolds, invariant subspaces, and orbifold lens spaces
motivate the definition of retraction sequences which 
we introduced in the previous section. 
For a vertex $v\in V(Q)$,  let $B_2$ be the 
union of all faces in $Q$ which does not contain $v$.  
Next, we consider a hyperplane 
\begin{equation}\label{eq_hyperplane}
H(v):=\{x \in \R^n \mid \langle x, p_v \rangle = q_v \} ,
\end{equation}
where $\langle~ ,~ \rangle$ denotes the Euclidean inner product,   
$p_v\in \R^n$ and $q_v\in \R$ are chosen in such a way that 
\begin{itemize}
\item $\{x \in \R^n \mid \langle x, p_v \rangle + q_v  \geq 0\}
	\cap V(Q)=\{v\}$,
\item $\{x \in \R^n \mid \langle x, p_v \rangle + q_v  \leq 0\}
	\cap V(Q)=V(Q) \setminus \{v\}$.
\end{itemize}
Then, $\Delta_Q(v): =Q\cap H(v)$ is an $(n-1)$-dimensional 
simplex, because $Q$ is a simple polytope of dimension $n$,
see Figure \ref{fig_geom_interpretation_of_ret_seq}. 

An $\mathcal{L}$-characteristic pair arises naturally from 
an $\mathcal{R}$-characteristic pair $(Q, \lambda)$ for each vertex $v$ of $Q$.
Indeed, if $v=F_{j_1} \cap \dots \cap F_{j_n}$, 
we denote the set of facets of $\Delta_Q(v)$ by 
$$ \mathscr{F}(\Delta_Q(v))= \{\Delta_Q(v) \cap F_{j_1}, 
\dots, \Delta_Q(v) \cap F_{j_n}\}.$$ 
Now we define a function 
\begin{equation}\label{eq_rat_char_ftn}
\xi_{Q,v} \colon \mathscr{F}(\Delta_Q(v))\to \Z^n,
\end{equation}
by $\xi_{Q,v}(\Delta_Q(v) \cap F_{j_r}) = \lambda(F_{j_r}),~ 
r=1,\dots, n$. Notice that $\dim \Delta_Q(v)=n-1$, but the rank 
of target space is $n$. 
Since $\{\lambda(F_{i_1}) \dots, \lambda(F_{i_n})\}$ is a linearly 
independent set, the function $\xi_{Q,v}$ is an $\mc{L}$-characteristic 
function on $\Delta_Q(v)$.

\section{Vanishing odd degree homology and torsion freeness}
\label{sec_van_odd_homology}
Now, we combine the ingredients which we introduced in the 
previous sections to derive a sufficient condition for vanishing 
odd degree cohomology of \quasitoric orbifolds. 
In particular, let 
$X(Q,\lambda)$ be a toric orbifold and the triple 
$\{(B_k, E_k, b_k)\}_{k=1}^\ell$ be a retraction sequence of $Q$.
Given an $n$-dimensional polytope $Q$, 
we begin by defining the following map
\begin{equation}\label{eq_h_v_straight_line_from_v}
h_{b_1} \colon \Delta_Q(b_1) \to B_2=
	\bigcup\{E \mid E \text{ is face of }Q, ~b_1\notin V(E)\}
\end{equation} by 
$h_{b_1}(x)= B_2\cap (\text{line passing through }x\text{ and }b_1)$, 
where $\Delta_Q(b_1)$ is an $(n-1)$-dimensional simplex. 
The map $h_{b_1}$ is well-defined, because $Q$ is convex. 
The left picture of Figure 
\ref{fig_geom_interpretation_of_ret_seq} shows the map $h_{b_1}$ when 
$Q$ is a prism.

\begin{figure}
\begin{tikzpicture}[scale=0.7]
	\draw[fill=blue!20] (21/4, 1.5)--(25/4,2)--(4,3)--cycle;
	\draw[fill=yellow!50, yellow!50] 
		(1/2,3)--(0,0)--(5,0)--(7,1)--(2,1)--cycle;
	\draw[thick,dotted,->] (1/2+5, 3)--(2.6,0.6);
	\draw[fill] (4.8, 4.8*5/6-19/12) circle [radius=0.05];
	\node[right] at (4.8, 4.8*5/6-21/12) {$x$};
	\node[right] at (2.8,0.5) {$h_{b_1}(x)$};
	\node at (1,1) {$B_2$};
	\draw[fill] (2.5,0.5) circle [radius=0.05];
	\draw (5,0)--(0,0)--(1/2,3);
	\draw (1/2,3)--(2,1)--(7,1)--(5,0);
	\draw[dashed] (1/2 +5,3)--(7,1);
	\draw[dashed] (1/2,3)--(1/2+5,3)--(5,0);
	\draw[dashed] (2,1)--(0,0);
	\draw[fill] (1/2+5, 3) circle [radius=0.05];
	\node[right] at (1/2 +5,3) {$b_1$};
	
	\draw[->] (6.5, 2.5)--(5.5, 4.8*5/6-2);
	\node[right] at (6.5,2.5) {$\Delta_Q(b_1)$};
	
	\begin{scope}[xshift=90mm]
	\draw[dashed] (0,0)--(5,0)--(7,1)--(2,1)--cycle;
	\draw[thick] (2,1)--(0,0)--(5,0);
	\draw[fill=yellow!50] (1/2,3)--(0,0)--(2,1)--cycle;
	\draw[thick] (6,1/2)--(5,1);
	\draw[thick, dotted, ->] (7,1)--(1.1,1/2);
	\node[above] at (1,1/2+0.1) {$h_{b_2}(x)$};
	\draw[thick, dotted, ->] (7,1)--(4.1,0.1);
	\node[below] at (4,0) {$h_{b_2}(y)$};
	\draw[fill] (7,1) circle [radius=0.05];
	\draw[fill] (4,0) circle [radius=0.05];
	\draw[fill] (1,1/2) circle [radius=0.05];
	\node[above] at (7,1) {$b_2$};
	\draw[fill] (37/7,6/7) circle [radius=0.05];
	\node[below] at (37/7,6/7) {$x$};
	\draw[fill] (29/5,3/5) circle [radius=0.05];
	\node[below] at (29/5,3/5) {$y$};
	
	\draw[->] (5,2) to [out=0,in=60] (5.1,0.95);
	\node[left] at (5,2) {$\Delta_{B_2}(b_2)$};
	\end{scope}	
\end{tikzpicture}
\caption{The geometric interpretation of a retraction sequence.}
\label{fig_geom_interpretation_of_ret_seq}
\end{figure}
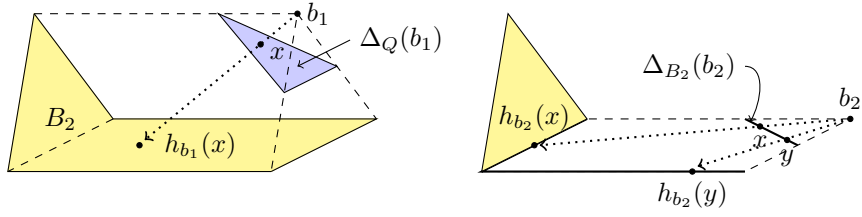

Define a map 
\begin{equation}\label{eq_map_before_quotient_by_eq_rel}
f_{b_1} \colon T^n \times \Delta_Q(b_1) \to \bigcup_{E:\text{ face of }
B_2}  T^{\dim E} \times E, 
\end{equation}
by $f_{b_1} (t,x)=( \bar{\rho}_E(t), h_{b_1}(x))$, 
where $\bar\rho_E$ is defined in 
\eqref{eq_induced_torus_map}. This induces the the map 
\begin{equation}\label{eq_map_lensspace_to_remaining}
	\bar{f}_{b_1} \colon L(\Delta_Q(b_1), \xi_{Q, b_1})  \to  
	\bigcup_{E:\text{face of }B_2}X(E, \lambda_E),
\end{equation}
where $\xi_{Q, b_1}$ is an $\mathcal{L}$-characteristic function 
defined in \eqref{eq_rat_char_ftn}. 
This map is well-defined from the proof of the following 
proposition.

\begin{proposition}\label{prop_cofibration}
The following diagram commutes:
\begin{equation}\label{eq_cofiber_seq}  
\xymatrix{T^n \times \Delta_Q(b_1) \ar[r]^-{f_{b_1}} 
\ar[d]^{/\sim_{\xi_{Q, b_1}}}  & 
\bigcup_{E:\text{ face of }B_2} ( T^{\dim E} \times E)
\ar[d]^{/ \sim_{\lambda_E}}&\\
L(\Delta_Q(b_1), \xi_{Q, b_1}) \ar[r]^-{\bar{f}_{b_1}}& 
\bigcup_{E:\text{ face of }B_2}X(E, \lambda_E) 
\ar@{^{(}->}[r] & X(Q, \lambda),  }
\end{equation}
where the equivalence relations $\sim_{\xi_{Q, b_1}}$ and $\sim_{\lambda_E}$ are
defined similarly as in \eqref{eq_equiv_rel_orb_lens_sp} and  
\eqref{eq_equiv_rel_of_toric_orb}, respectively. 
Moreover, the bottom row is a cofiber sequence. i.e., 
$X(Q,\lambda)$ is homotopy equivalent to 
the mapping cone $c(\bar{f}_{b_1})$ of the map $\bar{f}_{b_1}$.
\end{proposition}
\begin{proof}
	We first show that the map $\bar{f}_{b_1}$ is well-defined. 
	Suppose we choose two different representatives, 
	say $[t,x]_{\sim_{\xi_{Q, b_1}}}$ and $[s,y]_{\sim_{\xi_{Q, b_1}}}$ in 
	$L(\Delta_Q(b_1), \xi_{Q, b_1})$.  
	Then $x=y$, so $h_{b_1}(x)=h_{b_1}(y)$. Moreover, 
	if $x\in \Delta_Q(b_1) \cap F$ for some face $F$ of $Q$, 
	then $h_{b_1}(x) \in F \cap E$ 
	for some face $E$ of $B_2$. 
	Hence the map $\bar{\rho}_E$ sends the subtorus  $T_{F(x)}$  
	of $T^n$ to $T_{E(h_{b_1}(x))}$ the subtorus of $T^{\dim E}$. 
	Since the map $\bar\rho_E$ is a homomorphism, 
	if $t^{-1}s\in T_{F(x)}$, then 
	$$\bar{\rho}_E(t)^{-1}\bar{\rho}_E(s)=
	\bar{\rho}_E(t^{-1}s)\in T_{E(h_{b_1}(x))}.$$
	
	Let $C\Delta_Q(b_1)$ be the cone on $\Delta_Q(b_1)$ in $Q$ 
	with the cone point $b_1$. 
	Then, we can decompose $Q$ into two part as follows:
	\begin{equation}\label{eq_polytope=simplex_union_remaining}
	Q = C\Delta_Q(b_1) \cup_{\Delta_Q(b_1)} 
	\bar{Q \setminus C\Delta_Q(b_1)}.
	\end{equation}
	Now, we define a continuous surjective map 
	$$g_{b_1} \colon \bar{Q \setminus C\Delta_Q(b_1)} \to B_2$$ 
	in a manner similar to \eqref{eq_h_v_straight_line_from_v}. 
	We use it to define a straight line homotopy by
	$$\phi \colon \bar{Q \setminus C\Delta_Q(b_1)} \times I \to 
\bar{Q \setminus C\Delta_Q(b_1)},~~(x, u) \mapsto (1-u)x + u \cdot
 g_{b_1}(x),$$ 
which preserves the face structure. Thus, $\phi$ induces a homotopy 
$$\hat{\phi} \colon (T^n \times \bar{Q \setminus C\Delta_Q(b_1)})
/\sim_{\lambda}  \times I \to (T^n \times \bar{Q \setminus 
C\Delta_Q(b_1)})/\sim_\lambda$$ 
defined by 
$$([t, x]_{\sim_\lambda}, u) \mapsto [t, c(x, u)]_{\sim_\lambda}.$$
Note that at $u=0$ the map $\hat{\phi}$ is identity and at $u=1$ 
the image of $\hat{\phi}$ is $\pi^{-1}(B_2)$. 

Then, 
\begin{align*}
X(Q, \lambda)&=\pi^{-1}(C\Delta_Q(b_1)) \cup_{L(\Delta_Q(b_1), \xi_{Q, b_1})}
 \pi^{-1}( \bar{Q \setminus C\Delta_Q(b_1)})\\
&\simeq C\big(L(\Delta_Q(b_1),\xi_{Q, b_1} )\big) 
\cup_{L(\Delta_Q(b_1), \xi_{Q, b_1})} \pi^{-1}({B_2})\\
&\simeq c(\bar{f}_{b_1}).
\end{align*}
Hence, the result follows. 
\end{proof}

Now the following isomorphisms are straightforward from the 
cofiber sequence. 
\begin{align*}
	H_\ast(X(Q,\lambda), \pi^{-1}(B_2))&\cong 
	H_\ast(C(L(\Delta_Q(b_1), \xi_{Q, b_1})), \pi^{-1}(B_2))\\
	& \cong \widetilde{H}_{\ast-1}(L(\Delta_Q(b_1), \xi_{Q, b_1})).
\end{align*}
Those two isomorphisms come from the excision and the long exact 
sequence of the pair, respectively. 

So far, we have considered the $B_1(=Q)$ and $B_2$ which is the 
second term of a retraction sequence starting by choosing 
$b_1 \in FV(Q)=V(Q)$. However,   
we can apply the similar arguments to each pair $B_i$ and $B_{i+1}$ 
in a retraction sequence. This  
leads us the following Lemma whose proof is essentially same as 
that of Proposition \ref{prop_cofibration}. 
Before we state the lemma, we first set up the notations: 
Given a retraction sequence 
$\{(B_k, E_k, b_k)\}_{k=1}^\ell$ of $Q$, 
\begin{itemize}
\item $\Delta_{E_k}(b_k):=E_k \cap H(b_k)=B_k \cap H(b_k)$: the simplex obtained by cutting the vertex 
$b_k$ from $B_k$.
\item $\xi_{E_k, b_k}$: an $\mathcal{L}$-characteristic function on $\Delta_{E_k}(b_k)$ defined in a similar manner to \eqref{eq_rat_char_ftn}
induced from $\lambda_{E_k}$. 
\item The map 
$$\bar{f}_{b_k} \colon L(\Delta_{E_k}(b_k), \xi_{E_k, b_k}) \to \bigcup_{E: 
\text{face of } B_{k+1}} X(E, \lambda_E)=\pi^{-1}(B_{k+1})$$
is defined similarly to \eqref{eq_map_lensspace_to_remaining} by 
regarding $E_k$ as a simple polytope. 
\end{itemize}
The right hand side of Figure \ref{fig_geom_interpretation_of_ret_seq}
illustrates the case of the $3$-dimensional prism.
The argument above extends to prove the following lemma.
\begin{lemma}\label{lem_cofiber_seq}
	The sequence  
	\begin{equation}\label{eq_cofiber_seq_from_ret_seq}
	\xymatrix{L(\Delta_{E_k}(b_k), \xi_{E_k, b_k}) \ar[r]^-{\bar{f}_{b_k}}& 
	\pi^{-1}(B_{k+1}) \ar@{^{(}->}[r] & \pi^{-1}(B_k)  } 
	\end{equation}
	is a cofiber sequence. Moreover, 
$H_\ast(\pi^{-1}(B_k), \pi^{-1}(B_{k+1})) \cong \widetilde{H}_{\ast-1}
(L(\Delta_{E_k}(b_k), \xi_{b_k}))$. 
\end{lemma}

Recall from the Remark \ref{rmk_orb_lens_sp_finite_gp} that an 
$\mathcal{L}$-characteristic function $\xi \colon \mathscr{F}(\Delta^{n-1}) \to \Z^n$
defines a finite abelian group $\Z^n/ \text{im}(\xi)$. 
An $\mathcal{R}$-characteristic pair $(Q, \lambda)$ induces an 
$\mathcal{R}$-characteristic pair $(E, \lambda_E)$ as in \eqref{eq_def_of_induced_char_ftn}
for any face $E$ of $Q$. 
%
Let $E$ be a $k$-dimensional face of $Q$ ($k \leq n$) and $v\in V(E)$. 
Then, $\Delta_E(v):= E \cap H(v)$ is a $(k-1)$-simplex.
These give us an $\mathcal{L}$-characteristic function
$$\xi_{E, v}\colon \mathscr{F}( \Delta_E(v)) \to \Z^{k} $$
which is defined in a similar manner to \eqref{eq_rat_char_ftn}
associated to $\lambda_E \colon \mathscr{F}(E) \to \Z^k$ and $v \in V(E)$.
This $\mathcal{L}$-characteristic function defines the finite group
\begin{equation}\label{eq_def_of_G_E(v)} 
G_E(v):= \Z^k / \text{im} (\xi_{E,v}).
\end{equation}
If  $G_E(v)$ is trivial, we call a point $\pi^{-1}(v)$ in 
$\pi^{-1}(E)\cong X(E, \lambda_E)$ a \emph{smooth} point, otherwise a 
\emph{singular} point, 
where $\pi \colon X(Q, \lambda) \to Q $ is the orbit map defined in \eqref{eq_orbit_map}. 

Furthermore, for each $B\in \mathfrak{B}(Q)$ and a free vertex $v\in FV(B)$, 
there exists a unique maximal face, say $E_v$, of $B$ containing $v$. 
Hence, for each  $B\in \mathfrak{B}(Q)$, we denote by 
\begin{equation}\label{eq_def_G_B(v)}
G_B(v) := G_{E_v}(v), 
\end{equation}
whenever $v$ is a free vertex in $B$. 

\begin{proposition}\label{prop_local_group_subgroup}
	Given a vertex $v\in V(Q)$, let $E$ and $E'$ be two faces 
	containing $v$ such that $E$ is a face of $E'$. Then, 
	$|G_{E}(v)|$ divides  $|G_{E'}(v)|$. 
\end{proposition}
\begin{proof}
From Proposition \ref{prop_inv_subsp_is_toric_orb}, 
we may assume that $E'=Q$ without loss of generality.  
Suppose that $E$ is a face of $Q$ with codimension $k$. 
For convenience, we further assume that $E=F_1\cap \dots \cap F_{k}$
and $v=F_1 \cap\dots \cap  F_k \cap F_{k+1} \cap \dots \cap  F_n$, where 
$F_i$'s are facets of $Q$. 

From \eqref{eq_rat_char_ftn} and \eqref{eq_def_of_G_E(v)}, we have 
$G_Q(v)=\Z^n / \langle \lambda(F_1), \dots, \lambda(F_n) \rangle$ and 
$G_E(v)=\Z^k / \langle \lambda_E(E\cap F_{k+1}), \dots ,\lambda_E(E \cap F_n)\rangle$. 
Now we consider the following composition 
$$\xymatrix{\Z^n \ar@{->>}[r]^-{\rho_E} & \Z^k \ar@{->>}[r] & 
\Z^k / \langle \lambda_E(E\cap F_{k+1}), \dots ,\lambda_E(E \cap F_n)\rangle, }$$
where the map $\rho_E$ is defined in \eqref{eq_def_of_rho_E} and 
the second map is the natural surjection determined by \eqref{eq_def_of_induced_char_ftn}. 
Observe that the kernel of the previous composition contains $\langle \lambda(F_1), \dots, \lambda(F_n) \rangle$. 
Hence, we get a surjective  group homomorphism from $G_Q(v)$ to $G_E(v)$. 
The result follows from the Lagrange's theorem in group theory. 
\end{proof}

We are now in a position to prove Theorem \ref{thm_gcd_theorem}.
%

\begin{proof}[Proof of Theorem \ref{thm_gcd_theorem}]
We prove the claim by the induction on the number of vertices of $B \in \mathfrak{B}(Q)$. 
First, notice that when the retraction sequence reaches an edge or a  
union of edges, say $B_s$, then $\pi^{-1}(B_s)$ is $\CP^1$ or homotopic to a finite wedge of 
$\CP^1$, which implies that $H_\ast(\pi^{-1}(B_s))$ is torsion free and 
concentrated in even degrees. Therefore, if $|V(B)| \leq 2$ for $B\in \mathfrak{B}(Q)$, then 
the claim is true. 
	
Now we assume that $\pi^{-1}(B)$ is even for $B\in \mathfrak{B}(Q)$ with  $|V(B)|\leq i-1$. 
To complete the induction, we shall prove that the same holds for $B'\in \mathfrak{B}(Q)$ with  $|V(B')|=  i$.
Given such $B'$, there exist $B\in \mathfrak{B}(Q)$ such that $B$ is obtained from $B'$ by deleting 
all faces containing a free vertex of $B'$. 
To be more precise, let 
$FV(B')=\{v_{i_1}, \dots, v_{i_r}\}$ be the set of free vertices in 
$B'$. Notice that regarding $B'$ as a generic step of a retraction 
sequence in $\mathfrak{R}(Q)$, 
we can produce $r$  many different $B\in \mathfrak{B}(Q)$ with $|V(B)|=i-1$ from $B'$. 
According to the induction hypothesis, we assume that for each  $t=1,\dots, r$, the group 
$H_\ast(\pi^{-1}(B(v_{i_t})))$ is concentrated in even degrees and 
torsion free, where $B(v_{i_t})\in \mathfrak{B}(Q)$ is  
obtained from $B'$ by deleting faces containing $v_{i_t}$. 
This assumption makes sense, because any retraction 
sequence reaches a union of edges. 
	
For simplicity, we fix the following notation: For each free vertex 
$v_{i_t} \in FV(B')$, 
\begin{itemize}
\item $X' :=\pi^{-1}(B')$, $\dim B'= d'=\frac{1}{2}\dim_\R X'$.
\item $X(v_{i_t}):=\pi^{-1}(B(v_{i_t}))$, 
$\dim B(v_{i_t})= d =\frac{1}{2}\dim_\R X(v_{i_t}).$
\item $L(v_{i_t}) := L(\Delta_{E_{i_t}}(v_{i_t}), \xi_{E_{i_t}, v_{i_t}})$, where 
$E_{i_t}$ denotes the maximal face of $B'$ containing $v_{i_t}$. 
\end{itemize}
Notice that $\dim L(v_{i_t})\leq 2d' -1$ and $d \leq d'$. 

Now, we consider the following long exact sequence of the homology 
for the pair 
$(X', X(v_{i_t}))=(\pi^{-1}(B'), \pi^{-1}(B(v_{i_t}))$:
\begin{equation}\label{eq_les_of_pair}
\xymatrix@R=.5pc
{\quad \cdots\quad  \ar[r]&H_{j+1}(X') \ar[r]& 
H_{j+1}(X', X(v_{i_t})) \ar[r]& \\
H_{j}(X(v_{i_t}))  \ar[r]&
H_{j}(X')\ar[r] &
H_{j}(X', X(v_{i_t})) \ar[r] & \cdots.  }
\end{equation}

Suppose that $j$ is odd. By the induction hypothesis and  
Lemma \ref{lem_cofiber_seq}, the sequence \eqref{eq_les_of_pair} 
becomes 
\begin{equation}\label{eq_when_j=odd}
\xymatrix{0 \ar[r] &H_j(X') \ar[r] & 
\widetilde{H}_{j-1}(L(v_{i_t})) \ar[r]^-0& 
H_{j-1}(X(v_{i_t}))}.
\end{equation}
The map on the most right side is the zero map because 
the domain is a torsion group but the target space is free by 
assumption. 
Hence, $H_j(X')$ is isomorphic to $\widetilde{H}_{j-1}(L(v_{i_t}))$, 
and the latter is zero if $j-1>\dim L(v_{i_t})$ or 
a torsion group determined by the prime factors of 
$|G_{B'}({v_{i_t}})|$ if $j-1 \leq \dim L(v_{i_t})$ 
by Lemma \ref{lem_homology_orb_lens_sp}. This argument 
holds for each free vertex $v_{i_1}, \dots, v_{i_r}$. 
Hence we have $r$ many different exact sequences like 
\eqref{eq_when_j=odd}. 
Now, the assumption of Theorem \ref{thm_gcd_theorem} tells us that 
$${\rm gcd} \big\{ |\widetilde{H}_{j-1}(L(v_{i_1}))|, \dots, |\widetilde{H}_{j-1}
(L(v_{i_r}))| \big\}=1,$$
but $H_j(X')$ stays same. Hence, we conclude that $H_j(X')=0$ if $j$ is odd. 
Moreover, $\widetilde{H}_{j-1}(L(v_{i_t}))=0$ for all $t=1, \dots, r$ because of the 
exactness of \eqref{eq_when_j=odd}. 
	
Next, we assume that $j$ is even. Then, the exact sequence 
\eqref{eq_les_of_pair} 
gives us 
\begin{equation}\label{eq_when_j=even}
\xymatrix{\widetilde{H}_j(L(v_{i_t})) \ar[r]^-0& H_j(X(v_{i_t})) \ar[r]& 
H_j(X') \ar[r]& \widetilde{H}_{j-1}(L(v_{i_t})) \ar[r]& 0}.
\end{equation}
Then, we have  the following three cases:
$$\begin{array}{ll}
\cdots \overset{0}{\longrightarrow} H_j(X(v_{i_t})) 
			\longrightarrow H_j(X') 
			\longrightarrow 0, &  \text{ if } j-1>\dim L(v_{i_t}),\\
\cdots \overset{0}{\longrightarrow} H_j(X(v_{i_t})) 
			\longrightarrow H_j(X') 
			\longrightarrow \Z 
			\longrightarrow 0, & \text{ if } j-1=\dim L(v_{i_t}), \\
\cdots \overset{0}{\longrightarrow} H_j(X(v_{i_t})) 				
\longrightarrow H_j(X') 				
	\longrightarrow G_{j-1} 
	\longrightarrow 0, & \text{ if } j-1<\dim L(v_{i_t}),  
	\end{array}$$
where $G_{j-1}$ is defined in Lemma \ref{lem_homology_orb_lens_sp} and 
$H_j(X(v_{i_t}))$ is free by the induction hypothesis. 
The free vertices $v_{i_1}, \dots, v_{i_r}$  in $B'$ gives us 
$r$ many exact sequences, and each of them is one of the above three cases. 
If one of the free vertices gives the first or the second type of exact sequence, 
then $H_j(X')$ cannot have a torsion subgroup because of the exactness.  
If all of the sequences are of the third type, then $H_j(X')$ has no torsion
because of the assumption of the theorem and and arguments similar to those 
used in the case when $j$ is odd. This completes the induction. 
%
\end{proof}

Notice that Kawasaki in \cite{Ka} has shown that the cohomology ring of 
weighted projective space $\CP^n_{\chi}$ with weight 
$\chi=(\chi_0, \dots, \chi_n)$ is concentrated in even degrees and 
torsion free, if $\text{gcd}(\chi_0, \dots, \chi_n)=1$. Theorem
\ref{thm_gcd_theorem} extends Kawasaki's theorem to the category of 
toric orbifolds which contains the weighted projective spaces. 
The following Example 
\ref{ex_orbifold_over_polygon} shows how we can apply this result to a 
polygon, and Example \ref{ex_CP^4(2,2,2,1,1)} is a practical 
computation on a higher dimensional weighed projective space. 

\begin{example}\label{ex_orbifold_over_polygon}
Consider the 4-dimensional toric orbifold $X$ over $Q$ 
whose $\mathcal{R}$-characteristic pair is described in Figure 
\ref{fig_orbifold_over_polygon}.
Let $H(v)$ be an affine hyperplane defined in \eqref{eq_hyperplane}. 
Then $H(v)\cap Q$
is an $1$-simplex. The induced $\mathcal{L}$-characteristic function 
$$\xi_{Q,v} \colon \{ H(v)\cap F_1, H(v)\cap F_m\} \to \Z^2$$
is defined by 
$\xi_{Q,v} ( H(v)\cap F_1)=\lambda(F_1)=(a_1,b_1)$ and 
$\xi_{Q,v} ( H(v)\cap F_m)=\lambda(F_m)=(a_m,b_m)$. 
Therefore, the orbifold lens space $L(\Delta_Q(v), \xi_{Q,v})$ is homeomorphic 
to $S^3 \big/ G_Q(v)$, where $G_Q(v)$ is a finite abelian group of 
order $|a_1b_m-b_1a_m|$, see Proposition 
\ref{prop_orb_lens_sp=sphere_finite_gp}.
Moreover, the prime factors of the order of a torsion element in 
$H_\ast(L(\Delta_Q(v), \xi_{Q,v}))$ is a subset of the prime factors 
of $|a_1b_m-b_1a_m|$ by Lemma \ref{lem_homology_orb_lens_sp}. 
	
Now, we consider a retraction sequence 
$\{B_k, E_k, b_k\}_{k=1}^\ell$ starting at $v$. The second space 
$B_2$ is the union $F_2 \cup \dots \cup F_{m-1}$ of edges whose 
preimage $\pi^{-1}(B_2)$ is homotopic to the wedge of $m-2$ copies of $\CP^1$. 
Hence,  $H_\ast(\pi^{-1}(B_2))$ is torsion free and $H_{odd}(\pi^{-1}
(B_2))$ vanishes. A cofibration  
$$\xymatrix{ L(\Delta_Q(v), \xi_{Q,v}) \ar[r] & \pi^{-1}(B_2) \ar[r]& X}$$
gives an isomorphism $H_j(X, \pi^{-1}(B_2)) \cong \widetilde{H}_{j-1}
(L(\Delta_Q(v), \xi_{Q,v}) ).$ 
Hence, the long exact sequence of pair $(X, \pi^{-1}(B_2))$ yields
$$\cdots \to \widetilde{H}_j(L(\Delta_Q(v), \xi_{Q,v})) \to
H_j(\pi^{-1}(B_2))\to H_j(X) \to \widetilde{H}_{j-1}(L(\Delta_Q(v), \xi_{Q,v})) \to 
\cdots,$$
and this shows that, if $H_j(X)$ has a torsion part, then its prime factors must 
divide $|a_1b_m-b_1a_m|$.  
But, the same argument can be applied to all the other vertices in $Q$. Finally, 
we may conclude that $H_\ast(X)$ is torsion free and concentrated in 
even degrees, if 
\begin{equation}\label{eq_collection_of_loc_gp_ind_in_polygon}
{\rm gcd}\{ |a_1b_2-b_1a_2|, \dots,  
|a_{m-1}b_m-b_{m-1}a_m|, |a_1b_m-b_1a_m|\}=1,
\end{equation}
which is the assumption of Theorem \ref{thm_gcd_theorem}. 
\end{example}

\begin{figure}
\begin{tikzpicture}[scale=0.3]
\draw[top color=blue!40!white, opacity=0.4, white] 
(-1,1.5)--(-1, 3)--(0,6)--(3,7)--(6,6)--(7,3)--(7,1.5)--(-1,1.5) --cycle;
\draw (-1,1.5)--(-1, 3)--(0,6)--(3,7)--(6,6)--(7,3)--(7,1.5);
\node at (3,3.5) {\small$Q=m\text{-gon}$};
\node[left] at (-1, 4.5/2) {$\vdots$};
\node[rotate=70] at (-1.3, 4.7) {\scriptsize$F_2$};
\node[above, rotate=15] at (3/2, 13/2) {\scriptsize$F_1$}; 
\node[rotate=343] at (4.8, 7.2) {\scriptsize$F_m$};
\node[rotate=288] at (7.2, 4.5) {\scriptsize$F_{m-1}$};
\node[right] at (7, 4.5/2) {$\vdots$};
\node at (20,4) {$\Lambda=\begin{blockarray}{cccc}
	\lambda_1 & \lambda_2 & \cdots & \lambda_4 \\
	\begin{block}{[cccc]}
	a_1&a_2&\cdots & a_m\\ b_1& b_2&\cdots & b_m \\ 				
	\end{block}
	\end{blockarray}$};
	\node[above] at (3,7) {\small$v$};
	\draw[fill] (3,7) circle [radius=0.15]; 	
\end{tikzpicture}
\caption{An $\mathcal{R}$-characteristic function on a polygon.}
\label{fig_orbifold_over_polygon}
\end{figure}
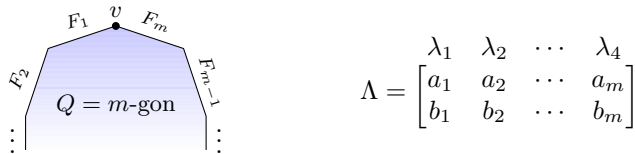

\begin{example}\label{ex_CP^4(2,2,2,1,1)}
We consider an $\mathcal{R}$-characteristic pair $(\Delta^4, \lambda)$, where  
$\lambda\colon \mathscr{F}(\Delta^4) \to \Z^4$ is defined by 
$$ \begin{blockarray}{ccccc}	
\lambda_1 & \lambda_2 & \lambda_3 & \lambda_4  & \lambda_5\\
\begin{block}{[ccccc]}
-1 &1 &0& 0& 0  \\ 
-2& 0& 1& 0& 0  \\ 
-2& 0& 0& 1& 0 \\ 
-2& 0& 0& 0& 1 \\
\end{block}
\end{blockarray}.$$	
The column vectors satisfies the relation  
$\lambda_1 + \lambda_2 + 2 \lambda_3 + 2\lambda_4 + 2\lambda_5=\mathbf{0}.$
Then the resulting toric orbifold is a weighted projective space $\CP^4_{(1,1,2,2,2)}$. 
We refer to \cite[Section 2.2]{Ful} or \cite[Example 3.1.17]{CLS} for more details. 

To check the assumption in Theorem \ref{thm_gcd_theorem}, it suffices to
consider all faces of $\Delta^4$, dimension greater than $1$, 
because the set $\mathfrak{B}(\Delta^4)$ coincides with the set of 
all faces of $\Delta^4$. 
First of all, for $\Delta^4$ itself, it is easy to see that 
$${\rm gcd} \big\{ |G_{\Delta^4}(v)| ~\big|~ v \in V(\Delta^4)\big\}=
{\rm gcd} \{ 2,2,2,1,1\}=1 $$
Since the process is essentially same, we choose 
$E=F_1\cap F_2=\Delta^2$ as a sample. Observe that 
\begin{align*}
\big(\langle \lambda_1,~\lambda_2 \rangle \otimes_\Z \R\big) \cap \Z^4
&=\big(\langle -e_1-2e_2-2e_3-2e_4, e_1 \rangle \otimes_\Z \R\big) 
\cap \Z^4\\
&\cong\langle e_2+e_3+e_4, e_1 \rangle. 
\end{align*}
Hence, we may decompose the target space 
$\Z^4\cong  \langle e_2+e_3+e_4 \rangle \oplus \langle e_1 \rangle 
\oplus \langle e_3 \rangle\oplus \langle e_4\rangle.$
This derives an $\mathcal{R}$-characteristic function 
$$\lambda_E \colon \{ E\cap F_3 , E\cap F_4, E\cap F_5\} \to \Z^2 
\cong \langle e_3 \rangle\oplus \langle e_4\rangle, $$
defined by  $\lambda_E(E\cap F_3)=(-1,-1), ~\lambda_E(E\cap 
F_4)=(1,0)$ and $\lambda_E(E\cap F_5)=(0,1)$. 
Hence, $\pi^{-1}(E)=X(\Delta^2, \lambda_E)\cong \CP^2_{(1,1,1)}$. 
Hence, we have 
$${\rm gcd}\big\{   |G_E(v)|  ~\big|~  v\in V(E) \big\} 
={\rm gcd} \{ 1,1,1\}=1.$$ 
\end{example}

Sometimes, if the polytope has sufficiently many symmetries, 
we can analyze all possible retraction sequences 
efficiently. Proposition \ref{prop_local_group_subgroup}
can then be used to ensure  the gcd assumption of Theorem 
\ref{thm_gcd_theorem} holds. 
The main features of the 
following example are that the polytope has at least $2$ free vertices at each 
$B\in \mathfrak{B}(Q)$, and that the collection $\{|G_Q(v)| \mid v\in V(Q)\} $ 
consists of mutually different prime numbers, in particular, they are pairwise relatively prime.

\begin{example}
Let $Q$ be the $3$-dimensional cube whose
facets and vertices are illustrated in Figure \ref{fig_prism_char_ftn}. 
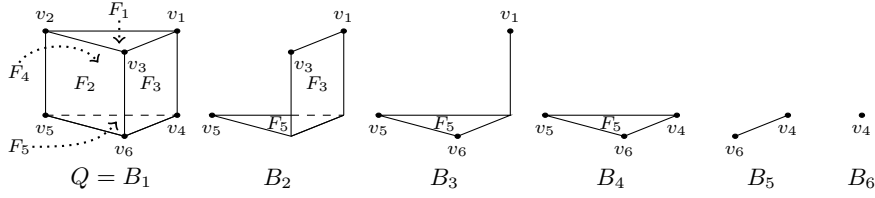
\begin{figure}
\begin{tikzpicture}[scale=0.35, yscale=0.8]
	\draw (0,0)--(3,-1)--(5,0);
        	\draw (5,0)--(5,4)--(0,4)--(0,0);
	\draw (0,4)--(3,3)--(3,-1)--(0,0);
	\draw (5,4)--(3,3);
	\draw (5,0)--(3,-1);
	\draw[dashed] (0,0)--(5,0);
	\node[above] at (0,4) {\scriptsize$v_2$};
	\draw[fill] (0,4) circle [radius=0.1];

	\node[above] at (5,4) {\scriptsize$v_1$};
	\draw[fill] (5,4) circle [radius=0.1];

	\node[below] at (3.5,3) {\scriptsize$v_3$};
	\draw[fill] (3,3) circle [radius=0.1];

	\node[below] at (0,0) {\scriptsize$v_5$};
	\draw[fill] (0,0) circle [radius=0.1];

	\node[below] at (5,0) {\scriptsize$v_4$};
	\draw[fill] (5,0) circle [radius=0.1];

	\node[below] at (3,-1) {\scriptsize$v_6$};
	\draw[fill] (3,-1) circle [radius=0.1];
	
	\node at (1.5,1.5) {\scriptsize$F_2$};
	\node at (4,1.5) {\scriptsize$F_3$};
	\node at (-1, 2) {\scriptsize$F_4$};
	\draw[thick, dotted, ->] (-1, 2.5) to [out=60, in=120] (2, 2.5);
	\node at (2.8, 5) {\scriptsize$F_1$};
	\draw[thick, dotted, ->] (2.8, 4.5)--(2.8, 3.4);
	\node at (-1, -1.5) {\scriptsize$F_5$};
	\draw[thick, dotted, ->] (-0.5, -1.5) to [out=0, in=240] (2.7, -0.3);
	\node at (2.5, -3) {\small$Q=B_1$};

	\begin{scope}[xshift=180]
	\draw (5,0)--(5,4)--(3,3)--(3,-1); 
	\draw (5,0)--(3,-1)--(0,0)--(3,0);
	\draw (5,0)--(3,-1);
	\draw[dashed] (3,0)--(5,0);
	\node at (4,1.5) {\scriptsize$F_3$};
	\node at (2.5,-0.4) {\scriptsize$F_5$};

	\node[above] at (5,4) {\scriptsize$v_1$};
	\draw[fill] (5,4) circle [radius=0.1];

	\node[below] at (3.5,3) {\scriptsize$v_3$};
	\draw[fill] (3,3) circle [radius=0.1];

	\node[below] at (0,0) {\scriptsize$v_5$};
	\draw[fill] (0,0) circle [radius=0.1];
	\node at (2.5, -3) {\small$B_2$};

	\begin{scope}[xshift=180]
	\draw (5,0)--(5,4);
	\draw (5,0)--(3,-1)--(0,0)--cycle;
	\node at (2.5,-0.4) {\scriptsize$F_5$};

	\node[above] at (5,4) {\scriptsize$v_1$};
	\draw[fill] (5,4) circle [radius=0.1];

	\node[below] at (0,0) {\scriptsize$v_5$};
	\draw[fill] (0,0) circle [radius=0.1];
	
	\node[below] at (3,-1) {\scriptsize$v_6$};
	\draw[fill] (3,-1) circle [radius=0.1];
	\node at (2.5, -3) {\small$B_3$};

\begin{scope}[xshift=180]
	\draw (5,0)--(3,-1)--(0,0)--cycle;
	\node at (2.5,-0.4) {\scriptsize$F_5$};

	\node[below] at (0,0) {\scriptsize$v_5$};
	\draw[fill] (0,0) circle [radius=0.1];
	
	\node[below] at (3,-1) {\scriptsize$v_6$};
	\draw[fill] (3,-1) circle [radius=0.1];
	
	\node[below] at (5,0) {\scriptsize$v_4$};
	\draw[fill] (5,0) circle [radius=0.1];
	\node at (2.5, -3) {\small$B_4$};
	
	\begin{scope}[xshift=120]
	\draw (5,0)--(3,-1);

	\node[below] at (3,-1) {\scriptsize$v_6$};
	\draw[fill] (3,-1) circle [radius=0.1];
	
	\node[below] at (5,0) {\scriptsize$v_4$};
	\draw[fill] (5,0) circle [radius=0.1];
	\node at (4, -3) {\small$B_5$};
\begin{scope}[xshift=80]
	\draw (5,0);
	\node[below] at (5,0) {\scriptsize$v_4$};
	\draw[fill] (5,0) circle [radius=0.1];
	\node at (5, -3) {\small$B_6$};

\end{scope}
\end{scope}	
\end{scope}
\end{scope}
\end{scope}
\end{tikzpicture}
\caption{A retraction sequence of a prism}
\label{fig_prism_char_ftn}
\end{figure}
We assign an $\mathcal{R}$-characteristic function $\lambda\colon \mathscr{F}(Q)\to \Z^3$ as follows;
\begin{align*}
\lambda(F_1) = (p_1, p_2, p_3),~\lambda(F_5)= (p_4, p_5, p_6),~
\lambda(F_2) = \mathbf{e}_1,~ \lambda(F_3) = \mathbf{e}_2,~\lambda(F_4) = \mathbf{e}_3,
\end{align*}
where $p_i$'s are all prime numbers with $p_i \neq p_j$ whenever $i\neq j$, and
$\mathbf{e}_i$ is the $i$-th standard unit vector in $\Z^3$.  
Then, it is easy to see that 
$|G_Q(v_i)|= p_i,$ for $i=1, \dots, 6.$
Hence, we have 
$$\gcd \big\{ |G_Q(v)| ~\big| ~ v\in V(Q) \big\}=\gcd \{p_1, \dots, p_6\} =1.$$ 
The same property holds for other polytopal complex $B\in \mathfrak{B}(Q)$
from Proposition \ref{prop_local_group_subgroup}. 
Indeed, for instance, 
$$\gcd  \big\{ |G_{B_2}(v)| ~\big| ~ v\in FV(B_2) \big\}
=\gcd \big\{ |G_{B_2}(v_1)|,~ |G_{B_2}(v_3)|,~ |G_{B_2}(v_5)| \big\}=1$$
because $\gcd\{p_1, p_3, p_5\}=1$. 
\end{example}

\section{Cohomology ring of toric orbifolds}
\label{sec_cohom_ring_and_ppoly}

The integral equivariant cohomology ring of certain 
projective toric varieties is given by a ring determined 
by the fan data. This ring is called the ring of 
\emph{piecewise polynomials} which we denote by 
$\mathcal{PP}[\Sigma]$. For a smooth fan, it uses the 
fan's combinatorial data only and coincides with the 
\emph{Stanley--Reisner ring} $\mathcal{SR}[\Sigma]$
of the fan $\Sigma$. 	In general however, the ring
of piecewise polynomials uses all the geometric data 
in a fan. 

To be more precise, let $\Sigma$ be a fan in $\R^n$ and 
$\{\lambda_1, \dots, \lambda_m\}\subset \Z^n$  the 
set of primitive vectors generating $1$-dimensional rays in $\Sigma$. Then, 
the Stanley--Reisner ring  $\mathcal{SR}[\Sigma]$ is 
defined by the quotient $\Z[x_1, \dots, x_m]/\mathcal{I}$
of polynomial ring with $m$-variables by the following 
ideal generated by square free monomials:
\begin{equation}\label{eq_SR-ideal}
\mathcal{I}=\left< x_{i_1}\cdots x_{i_k} \mid 
cone\{\lambda_{i_1}, \dots, \lambda_{i_k} \} \notin
\Sigma \right>
\end{equation}
where $cone\{\lambda_{i_1}, \dots, \lambda_{i_k} \}$ 
denotes the cone generated by $\{\lambda_{i_1},\dots, \lambda_{i_k}\}$. 
For the case of smooth toric varieties, their odd degree 
cohomology always vanishes, which leads us 
the following description of cohomology ring.
\begin{theorem}[\cite{Dan}, \cite{Jur}]
Let $X_\Sigma$ be a smooth toric variety. Then, 
there exists a ring isomorphism 
$H^\ast(X_\Sigma)\cong 
\mathcal{SR}[\Sigma]/\mathcal{J}$, 
where $\mathcal{J}$ is the ideal generated by the 
linear relations
\begin{equation}\label{eq_linear_relations}
\sum_{i=1}^m \langle \lambda_i, \mathbf{e}_j\rangle x_i=0, 
\quad  j=1, \dots, n,
\end{equation}
where $\mathbf{e}_j$ denotes the $j$-th standard unit vector in $\Z^n$. 
\end{theorem}
Notice that, for toric orbifolds, the theorem holds only 
for $\Q$-coefficients; see for instance, 
\cite[Section 12.4]{CLS}. 
In order to make the singular 
theory better resemble the smooth case, we introduce 
an intermediate ring, which models the Stanley--Reisner 
ring but is based on a fan $\widehat{\Sigma}$ in $\R^m$ 
defined from the combinatorial data of $\Sigma$, 
which has $m$ one-dimensional rays. The ring of 
piecewise polynomials on the original fan $\Sigma$ is 
recovered by imposing an \emph{integrality condition}, 
which leads us the notion of the 
\emph{weighted Stanley--Reisner ring} 
$w\mathcal{SR}[\Sigma]$ of $\Sigma$. 
		
\subsection{Weighted Stanley--Reisner ring}
Let $\Sigma$ be a simplicial fan in $\R^n$, i.e.,  
each top dimensional cone of $\Sigma$ is generated by 
$n$ linearly independent primitive vectors  in the lattice $\Z^n$.
In particular, a simplicial fan $\Sigma$ is called a \emph{polytopal fan}
if it is the normal fan of a simple lattice polytope in $\R^n$;
see \cite[Chapter 2]{CLS} or 
\cite[Section 1.5]{Ful} for more details. 
Hence, the determinant of generators of each top 
dimensional cone is nonzero but not necessarily be $\pm 1$, 
and so the corresponding fixed point might be 
singular.  Let $\Sigma^{(j)}$ denotes the set of $j$-dimensional cones in $\Sigma$. 
To record the singularity of each fixed point in an 
efficient way, we assign a vector 
$$z^\sigma:= (z^\sigma_1, \dots, z^\sigma_m) \in
 \bigoplus_m \Q[u_1, \dots, u_n]$$
to each top dimensional cone 
$\sigma=cone\{\lambda_{i_1}, \dots, \lambda_{i_n}\}
\in \Sigma^{(n)}$ 
by the following rule:
\begin{enumerate}
\item[(C1)] $z^\sigma_j=0$ if $j\notin \{i_1, \dots, i_n\}$,
\item[(C2)] 
$\begin{bmatrix} 
z^\sigma_{i_1} \\ \vdots \\ z^\sigma_{i_n}
\end{bmatrix}=
\left[ \begin{array}{c|c|c} &&\\ \lambda_{i_1} &\cdots &  
\lambda_{i_n}\\&& \end{array} \right]^{-1}
\cdot \begin{bmatrix} u_1\\ \vdots \\ u_n \end{bmatrix}$.
\end{enumerate}

The inverse matrix in the condition (C2) may have 
rational entries. The following 
definition is motivated by this observation. 
\begin{definition}\label{def_int_cond}
Given a fan $\Sigma$ in $\R^n$ with 
$m$ one-dimensional rays, we say a polynomial 
$h(x_1, \dots, x_m) \in \Z[x_1, \dots,  x_m]$ satisfies the 
\emph{integrality condition} with respect to $\Sigma$, if 
$h(z^\sigma)\in \Z[x_1, \dots, x_m]$ for all 
$\sigma\in \Sigma^{(n)}$. 
\end{definition}

Notice that the collection of polynomials satisfying 
the integrality condition is closed under 
the addition and multiplication, which induces the 
natural ring structure on it inherited from that of $\Z[x_1, \dots, x_m]$. 
Moreover, the polynomials in 
$\mathcal{I}$ defined in \eqref{eq_SR-ideal} satisfy 
the integrality condition because of the condition (C1). 
Finally, we define the weighted Stanley--Reisner ring 
$w\mathcal{SR}[\Sigma]$ as follows:
\begin{equation}\label{def_wSR}
w\mathcal{SR}[\Sigma]:=\{h \in \Z[x_1, \dots, x_m] 
\mid h~ \text{satisfies the \emph{integrality condition} }\} / \mathcal{I}.	
\end{equation}

\begin{remark}
When the fan $\Sigma$ is smooth, 
$w\mathcal{SR}[\Sigma]=\SR[\Sigma]$. Indeed, 
the determinant of a smooth top dimensional cone is 
$\pm 1$, which implies that its inverse has integer 
entries.
\end{remark}

%

Now, we introduce the second main theorem of this paper. 
The proof will be given in the next subsection. 
\begin{theorem}\label{thm_cohomology_ring}
Let $X_\Sigma$ be a projective toric orbifold over 
a polytopal fan $\Sigma$
with $H^{odd}(X)=0$. Then, there is a ring isomorphism 
$$H^\ast(X_\Sigma)\cong w\SR[\Sigma]/\mathcal{J},$$ 
where $\mathcal{J}$ is the ideal generated by linear 
relations \eqref{eq_linear_relations}. 
\end{theorem}

Consider a simple lattice polytope $Q$ in $\R^n$ whose 
normal fan is $\Sigma$.
Then, the normal vectors of each facet define an 
$\mathcal{R}$-characteristic function 
$\lambda \colon \mathscr{F}(Q) \to \Z^n$. 
Now, we have a natural $\mathcal{R}$-characteristic pair 
$(Q, \lambda)$ from $\Sigma$, which allows us to apply 
the results of Section \ref{sec_ret_of_simple_poly} 
and  Section \ref{sec_van_odd_homology}. 
Hence, we have a concrete statement which is Theorem 
\ref{thm_cohom_ring_without_Hodd=0}
with a sufficient condition for $H^{odd}(X_\Sigma)=0$.

We complete this subsection by applying Theorem 
\ref{thm_cohom_ring_without_Hodd=0} to a weighted
projective space $\CP^2_{(1,a,b)}$. We shall recover
the Kawasaki's result \cite[Theorem 1]{Ka}. 

\begin{example}\label{ex_wSR_for_CP^2(1,a,b)}
Let $\Sigma$ be a fan in $\R^2$ generated by 
\begin{equation}\label{eq_prim_vec_of_2-dim_wps}
\lambda_1=(a,b), ~\lambda_2=(-1,0), ~
\lambda_3=(0,-1) \in \Z^2 
\end{equation}
where $a$ and $b$ are relatively prime. 
The 2-dimensional cones are
$\sigma_{12}, \sigma_{13}, \sigma_{23}$, 
where $\sigma_{ij}=cone\{\lambda_i,\lambda_j\}$.
Since $\{\lambda_1, \lambda_2, \lambda_3\}$ 
generates the lattice $\Z^2$ 
and satisfies 
$\lambda_1 +a\lambda_2+b\lambda_3=(0,0)$, 
the toric variety $X_\Sigma$ is isomorphic to the 
weighed projective space $\CP^2_{(1,a,b)}$. 
We refer to \cite[Section 2.2]{Ful} or  \cite[Example 3.1.17]{CLS} for 
the characterization of a fan corresponding to 
weighted projective spaces. 

The direct computation of inverse matrices for 
$\left[ \begin{array}{c|c} \lambda_i& \lambda_j \end{array}\right]$
gives us the following list of vectors:
\begin{displaymath}
\begin{array}{rrrrrl}\vspace{0.1cm}
z^{\sigma_{12}}=&\big(& \frac{1}{b}u_2,&-u_1 +\frac{a}{b}u_2,&  0 &\big),\\ \vspace{0.1cm}
z^{\sigma_{13}}=&\big(&\frac{1}{a}u_1,& 0,&\frac{b}{a}u_1-u_2 &\big),\\ 
z^{\sigma_{23}}=&\big(&0, &-u_1, &-u_2 &\big).
\end{array}
\end{displaymath}
%
Hence, we have 
\begin{align}\label{eq_SR_for_CP_235_with_int_cond}
w\mathcal{SR}[\Sigma]=\left\{ h(x_1,x_2,x_3)\in \Z[x_1,x_2,x_3] \mid 
h(z^{\sigma_{ij}})\in \Z[u_1, u_2],~ \text{for}~  1\leq i<j\leq 3
\right\}
/\mathcal{I}.
\end{align}
Finding elements at each degree is straightforward. For instance, 
a degree 2 polynomial $k_1x_1+k_2x_2+ k_3x_3 \in w\mathcal{SR}[\Sigma]$
if and only if the following three polynomials have integer coefficients:
\begin{enumerate}
\item $-k_2u_1+\left(\frac{1}{b}k_1+\frac{a}{b}k_2\right) u_2,$
\item $\left( \frac{1}{a}k_1 + \frac{b}{a}k_3\right)u_1 -k_3u_2,$
\item $-k_2u_1-k_2u_2$,
\end{enumerate}
which is exactly the case when 
$k_1+ak_2 \in b\Z$ and $k_1+bk_3\in a\Z.$
Hence, one can show that the integers $(k_1, k_2, k_3)$ are 
$$(a,-1,0),~(b,0,-1),~ (ab,0,0),~(0,b,0),~(0,0,a),$$
and $\Z$-linear combinations of them. 
They  give us the following degree $2$ elements in $w\mathcal{SR}[\Sigma]$,
\begin{equation}\label{eq_deg_2_gen_in_wSR}
ax_1-x_2,~bx_1-x_3,~ abx_1,~bx_2, ~ax_3,
\end{equation}
and $\Z$-linear combinations of them. 
Similarly, we can find the degree $4$ elements by 
\begin{equation}\label{eq_deg_4_gen_in_wPS}
a^2b^2x_1^2,~b^2x_2^2,~ a^2x_3^2,~ abx_1x_2,~ a^2x_1x_3,~ x_2x_3,
\end{equation}
and $\Z$-linear combinations of them. 

We continue to calculate the ring structure of $H^\ast(\CP^2_{(1,a,b)})$
by Theorem \ref{thm_cohom_ring_without_Hodd=0}. Indeed, 
the $\mathcal{R}$-characteristic pair $(\Delta^2, \lambda)$ induced from 
$\Sigma$ satisfies the assumption of Theorem \ref{thm_gcd_theorem}, 
see Example \ref{ex_orbifold_over_polygon}. 
Hence, we conclude that $\CP^2_{(1,a,b)}$ is \emph{even}, which implies 
that the rank of integral cohomology group is $1$ in each even degree and 
$0$ otherwise. 

\begin{remark}\label{rmk_betti_num=h-vector}
In general, the integral betti numbers of a toric manifold  or 
the rational betti numbers of a toric orbifold are given by the $h$-vector 
of underlying polytope, see \cite[Section 3]{DJ} or \cite[Section 4]{PS}. 
Hence, if a toric orbifold is even, then its integral betti numbers are 
obtained by the $h$-vector of the underlying polytope. 
\end{remark}

Now, the characteristic vectors \eqref{eq_prim_vec_of_2-dim_wps} 
and the relation \eqref{eq_linear_relations} determine
the ideal $\mathcal{J}=\langle ax_1-x_2,~bx_1-x_3\rangle$ 
whose generators are first two items in \eqref{eq_deg_2_gen_in_wSR}. 
Hence, the elements in \eqref{eq_deg_2_gen_in_wSR} except first two are all eventually equal 
by $\mathcal{J}$ in $H^\ast(\CP^2_{(1,a,b)})$. We put 
$$w_1:=abx_1=bx_2=ax_3. $$
Since rank$H^4(\CP^2_{(1,a,b)})=1$, we choose an element in 
\eqref{eq_deg_4_gen_in_wPS} which has the minimal divisibility. 
In this case, we pick up 
$$w_2:=x_2x_3.$$ 
Then, we have the multiplicative structure $w_1^2=abw_2$. 
Finally, we have the following presentation
$$H^\ast(\CP^2_{(1,a,b)})\cong \Z[w_1, w_2]/\langle w_1^2-abw_2, w_1w_2\rangle,$$
where $\deg w_1=2,~\deg w_2=4$. Notice that the monomial $w_1w_2$ comes from 
the Stanley--Reisner ideal $x_1x_2x_3$.  	
\end{example}

\begin{remark}
Even if we can find elements in $w\SR[\Sigma]$ by the direct computation 
of integrality condition, finding the minimal set of generators
in $w\mathcal{SR}[\Sigma]$ for arbitrary simplicial fan is not obvious, in general. 
However, 
when $X_\Sigma$ is a weighted projective space, 
a result of \cite{BFR} allows us to find generators
of the ring of piecewise polynomials $\mathcal{PP}[\Sigma]$ 
and hence, generators in $w\SR[\Sigma]$, by a method in the 
next subsection. 
Moreover, the identification result Corollary \ref{cor_PP=wSR} 
tells us how to interprete those generators in terms of 
elements in $w\SR[\Sigma]$. 
\end{remark}


\subsection{Piecewise algebra and cohomology ring}
\label{subsec_p.algebra}
We introduce now the ring of piecewise polynomials 
which is determined by a fan and describes the 
equivariant cohomology of a large class of toric orbifolds. 
As mentioned above, unlike the Stanley--Reisner ring, 
which encodes combinatorial data only, the 
ring of piecewise polynomials depends the full 
geometric information in a fan. 

We begin by introducing piecewise polynomials.
Let $\Sigma$ be a fan in  $\R^n$.  A function 
$f \colon \Z^n \to \Z$  is called a 
\emph{piecewise polynomial} on $\Sigma$  
if for each cone $\sigma\in \Sigma$ the restriction 
$f|_\sigma$ is a polynomial function on $\sigma\cap \Z^n$. 
Such function can be interpreted as a collection 
$\{f_\sigma\}_{\sigma\in \Sigma^{(n)}}$, which we 
denote by $\{f_\sigma\}$ for simplicity,  such that
\begin{equation}\label{eq_property_of_p_poly}
f_\sigma |_{_{\sigma\cap \sigma'}} = 
f_{\sigma'} |_{_{\sigma\cap \sigma'}}.
\end{equation}
In other words, it is enough to consider the polynomials 
on each top dimensional cone. The polynomials on 
lower dimensional cones are determined by 
\eqref{eq_property_of_p_poly}. 	

The set $\PP[\Sigma]$ of piecewise polynomial functions on $\Sigma$
with integer coefficients on $\Sigma$ has a ring structure 
under the pointwise addition and multiplication. 
Moreover, the natural inclusion of global polynomials 
$\Z[u_1, \dots, u_n]$ into $\PP[\Sigma]$ 
induces a $\Z[u_1, \dots, u_n]$-algebra structure on 
$\PP[\Sigma]$. Furthermore, by considering $\Q^n$ 
instead of $\Z^n$, we can define piecewise polynomial 
functions with rational coefficients $f \colon \Q^n \to \Q$, 
and we denote the ring of piecewise polynomial functions 
with rational coefficients by $\PP[\Sigma;\Q]$.  
	
It is well-known that the equivariant cohomology ring with 
rational coefficients of  a toric variety  over a simplicial fan is isomorphic 
to $\PP[\Sigma;\Q]$, see \cite{CLS}. On the other hand, for the 
case of polytopal fans,
Bahri, Franz and Ray \cite{BFR}
proved the following proposition over $\Z$. 
\begin{proposition} \label{prop.H_T=PP_by_BRF} 
\cite[Proposition 2.2]{BFR}. Let $\Sigma$ be a polytopal 
fan in $\R^n$, $X_\Sigma$  the associated compact 
projective toric variety with $H^{odd}(X_\Sigma)=0$, 
and $T=T^n$ the $n$-dimensional torus acting on $X_\Sigma$.  
Then, $H^\ast_T(X_\Sigma)$ is isomorphic to 
$\PP[\Sigma]$ as an $H^\ast(BT)$-algebra.
\end{proposition}
Here, $H^\ast(BT)$-algebra structure on 
$\PP[\Sigma]$ is obtained by identifying $H^\ast(BT)$ 
with the global polynomials $\Z[u_1, \dots, u_n]$, where 
$u_i$ is the first Chern class of the canonical line 
bundle given by $i$-th projection $T\to S^1$. 
	 
On the other hand, the combinatorial structure of 
$\Sigma$ determines a canonical fan in a higher 
dimensional lattice as follows: 
Let $\Sigma^{(1)}=\{ \lambda_1, \dots, \lambda_m\} $ 
be the set of primitive vectors generating 
$1$-dimensional rays in $\Sigma$. We define  
a linear map  $\Lambda \colon  \Z^m \to \Z^n$ by 
$\Lambda(\mathbf{e}_i)=\lambda_i$, where 
$\mathbf{e}_1, \dots, \mathbf{e}_m$ 
denote the standard unit vectors in $\Z^m$. 
By the pull-back of $\Sigma$ through $\Lambda$, 
we can define a fan
$$\Sigmahat =\{ \hat{\sigma}:=\Lambda^{-1}(\sigma) 
\mid \sigma \in \Sigma\}$$
in $\R^m$. To be more precise, if $\sigma$ is the 
cone generated by $\lambda_{i_1}, \dots, \lambda_{i_k}$, 
then $\hat\sigma$ is the cone generated 
by $\mathbf{e}_{i_1}, \dots, \mathbf{e}_{i_k}.$
Moreover, for a commutative ring $\mathbf{k}$, 
a linear map $\Lambda$ induces a ring homomorphism
\begin{equation}\label{eq_pullback_of_Lambda}
\Lambda^\ast : \mathcal{PP}[\Sigma; \mathbf{k}] \to 
\mathcal{PP}[\Sigmahat;\mathbf{k}] 
\end{equation}
of piecewise polynomial rings, where the map is defined by 
\begin{equation*}\label{eq_Lambda*}
\Lambda^\ast(\{f_\sigma\})=\big\{
g_{\sigmahat} (x_{i_1}, \dots, x_{i_n}) := 
f_\sigma (\Lambda_\sigma \cdot  
[x_{i_1}, \dots, x_{i_n}]^T )\big\}_{\sigmahat \in \Sigmahat^{(n)}}
\end{equation*}
where $\Lambda_\sigma 
=\left[ \begin{array}{c|c|c} \lambda_{i_1}&\dots&
 \lambda_{i_n} \end{array}\right]$ is a square matrix and

  Indeed, the map $\Lambda^\ast$ is well-defined, since 
 $g_{\sigmahat}|_{_{\sigmahat \cap \sigmahat'}} = g_{\sigmahat'}|_{_{\sigmahat \cap \sigmahat'}}$.

\begin{lemma}\label{lem_relation_PPhat_PP}
Given a polytopal fan  $\Sigma$, 
as $H^\ast(BT;\mathbf{k})$-algebra:
\begin{enumerate}
\item When $\mathbf{k}=\Q$,  $\PP[\Sigma;\Q]$ 
is isomorphic to 
$\PP[\Sigmahat;\Q]$. 
\item When $\mathbf{k}=\Z$, there is a 
monomorphism from $\PP[\Sigma]$
to $\PP[\Sigmahat]$. 
\end{enumerate}
\end{lemma}
\begin{proof}
For each top dimensional cone 
$\sigma=cone\{\lambda_{i_1}, \ldots,
 \lambda_{i_n}\} \in \Sigma^{(n)}$, 
we set the following notation:
\begin{itemize}
\item $f_\sigma(u_1, \dots, u_n)$, 
$g_{\hat{\sigma}}(x_{i_1}, \dots, x_{i_n})$ 
: polynomial functions defined on $\sigma\in \Sigma$ 
and $\hat{\sigma}\in \Sigmahat$, respectively.
\item $\{f_\sigma\} :=\{f_\sigma(u_1, \dots, u_n) 
\mid \sigma \in \Sigma^{(n)}\} \in \PP[\Sigma]$.
\item $\{g_{\sigmahat}\} :=
\{g_{\sigmahat}(x_{i_1}, \dots, x_{i_n}) \mid 
\hat{\sigma} \in \Sigmahat^{(n)}\} \in \PP[\Sigmahat]$. 
\item $\Lambda_\sigma:=\left[\lambda_{i_1} 
\mid \cdots \mid \lambda_{i_n} \right]$ : 
$n \times n$ matrix with column vectors 
$\lambda_{i_1},\ldots,\lambda_{i_n}$.
\end{itemize}
Recall the ring homomorphism $\Lambda^\ast$ 
introduced  in \eqref{eq_pullback_of_Lambda}. 
If we restrict $\mathbf{k}$ to $\Q$, the map 
$\Lambda^\ast$ has the natural inverse
\begin{equation}\label{eq_Theta} 
\Theta \colon \PP[\widehat{\Sigma};\mathbf{k}] \to  
\PP[\Sigma;\mathbf{k}]
\end{equation} 
defined  by 
$$\Theta(\{g_{\sigmahat}\})=\{f_\sigma(u_1, \dots, u_n) :=
g_{\sigmahat} (\Lambda_\sigma^{-1} \cdot [u_1, \ldots, u_n]^T
\mid \sigma \in \Sigma^{(n)}\},$$
where $\Lambda_\sigma^{-1}$ is regarded as a linear automorphism of $\Q^n$.  
Indeed, 
\begin{align*}
            (\Theta\circ \Lambda^\ast)(\{f_\sigma\})&=
            \{ f_\sigma (\Lambda_\sigma \cdot   \Lambda_\sigma^{-1} \cdot [u_1,\ldots,u_n]^T)
             \mid \sigma \in \Sigma^{(n)} \}\\
             &=\{f_\sigma(u_1,\ldots, u_n) \mid \sigma \in \Sigma^{(n)}\}=\{f_\sigma\}.
        \end{align*}
	In particular, $\Lambda^\ast$ is a monomorphism in $\Z$-coefficients.  
	Finally, the $H^\ast(BT;\mathbf{k})$-algebra structure on $\PP[\Sigmahat ;\mathbf{k}]$ 
	is naturally inherited from that of $\PP[\Sigma;\mathbf{k}]$ 
	via the map $\Lambda^\ast$.
	\end{proof}
Recall that the Stanley--Reisner ring $\mathcal{SR}[\Sigma;\mathbf{k}]$ has combinatorial data only, 
while $\PP[\Sigma;\mathbf{k}]$ contains both combinatorial and geometric data. 
However, $\PP[\Sigmahat;\mathbf{k}]$ has only combinatorics, but looks 
like $\PP[\Sigma;\mathbf{k}]$. In this point of view, $\PP[\Sigmahat;\mathbf{k}]$ is an intermediate 
object between $\mathcal{SR}[\Sigma;\mathbf{k}]$ and $\PP[\Sigma;\mathbf{k}]$.  
The following lemma together with Lemma \ref{lem_relation_PPhat_PP} concludes 
the relations among those three objects. 
\begin{lemma}\label{lem_SR=PPhat}
As $H^\ast(BT;\mathbf{k})$-algebra, $\PP[\Sigmahat;\mathbf{k}]$ is isomorphic to $\mathcal{SR}[\Sigma;\mathbf{k}]$ 
for  $\mathbf{k}=\Z$ or $\Q$. 
\end{lemma}
\begin{proof}
We construct an isomorphism between $\PP[\Sigmahat;\mathbf{k}]$ and $\mathcal{SR}[\Sigma;\mathbf{k}]$, 
where $\mathbf{k}=\Z \text{ or } \Q$.  Assume that $|\Sigma^{(1)}|=m$. 
Define a map
\begin{equation}\label{eq_alpha}
    	\phi \colon \mathbf{k}[x_1,  \ldots , x_{m}] \to \PP[\widehat{\Sigma};\mathbf{k}]
\end{equation}
    by restriction to each cone of $\widehat{\Sigma}$. 
    Then, this map $\phi$ is surjective ring homomorphism. 
    Indeed, given $\{g_{\sigmahat} \} \in \PP[\widehat{\Sigma};\mathbf{k}]$, 
    we can  apply the \emph{inclusion-exclusion principle} to obtain 
\begin{equation}\label{eq_inc_exc_principle}
    h(x_1, \ldots , x_{m})=
    \sum_{j=0}^{n-1}\left((-1)^{j} \sum_{\substack{ \hat\tau \in \Sigmahat \\ \dim \hat{ \tau}=n-j}}g_{\hat\tau}(x_{i_1}, \dots, x_{i_{n-j}}) \right)
\end{equation}
 which is the desired global function $h$ satisfying  
    $\phi(h)=\{g_{\hat\sigma}\}$, where $\hat\sigma \in \Sigmahat^{(n)}$. 

Moreover, since the zero element in $\PP[\Sigmahat;\mathbf{k}]$ is 
	$\{  g_{\sigmahat} =0 \mid \hat\sigma\in \widehat\Sigma^{(n)} \}$, the kernel is 
	$$\ker \phi = {\rm span} \left\{\prod_{j=1}^{k}x_{i_j}  \mid 
	{\rm cone}\{\mathbf{e}_{i_1}, \dots, \mathbf{e}_{i_k} \}   \notin \widehat{\Sigma} \right\}, $$
	which is exactly the Stanley--Reisner ideal $\mathcal{I}$ of $\Sigma$. 
	Hence, the result follows. 
\end{proof}

\begin{corollary}\label{cor_PP=wSR}
There exists an isomorphism $\PP[\Sigma]\cong w\SR[\Sigma]$ 
(see \eqref{def_wSR}) as $H^\ast(BT)$-algebra. 
\end{corollary}
\begin{proof}
Cosider the composition of ring homomorphisms 
$$\xymatrix{\PP[\Sigma]~~\ar@{>->}[r]^-{\Lambda^\ast}& \PP[\Sigmahat]\ar[r]^-{\Phi^{-1}}\ar[r] & \SR[\Sigma]},$$
where $\Phi \colon  \SR[\Sigma] \to \PP[\Sigmahat] $ is the isomorphism induced by $\phi$. 
With $\Z$-coefficients, the map $\Lambda^\ast$ is injective by Lemma \ref{lem_relation_PPhat_PP}. 
Hence, $\PP[\Sigma]$ is isomorphic to its image in $\SR[\Sigma]$ via the composition $\Phi^{-1} \circ\Lambda^\ast$. 

Recall that the composition $\Phi^{-1}\circ\Lambda^\ast$ is isomorphism over $\Q$, whose 
inverse $\Theta\circ \Phi^{-1}$ maps an element $[h]\in \SR[\Sigma;\Q]$ to 
$\{h(z^\sigma)\}_{\sigma\in \Sigma^{(n)}}   \in \PP[\Sigma;\Q]$. Therefore, 
over integer coefficients, $[h] \in \text{im}(\Phi^{-1} \circ\Lambda^\ast)$ if and only if 
the polynomial $h$ satisfies the integrality condition. Hence, the result follows. 
\end{proof}

Finally, we conclude this subsection with a proof of Theorem \ref{thm_cohomology_ring}. 
	\begin{proof}[Proof of theorem \ref{thm_cohomology_ring}]
	Since $H^\ast(X_\Sigma;\Z)$ concentrated in even degrees, 
	the Serre spectral sequence for the fibration
	$$\xymatrix{X_\Sigma \ar[r] & ET\times_T X_\Sigma \ar[r]^-{\pi} & BT}$$
	degenerates at $E_2$ level. 
	By the result from Franz and Puppe (\cite{FrPu}, Theorem 1.1), 
	we get the following isomorphisms of $H^\ast(BT)$-algebras,
        \begin{align*}
         H^\ast(X_\Sigma) &\cong H^\ast_T(X_\Sigma)\otimes_{H^\ast(BT)}\Z \\
         &\cong H^\ast_T(X_\Sigma) / {\im }(\pi^\ast \colon H^\ast(BT) \to H^\ast_T(X_\Sigma)).
        \end{align*}
        By Proposition \ref{prop.H_T=PP_by_BRF} and Corollary \ref{cor_PP=wSR},  
        we have 
        $H^\ast_T(X_\Sigma) \cong w\mathcal{SR}[\Sigma]$. Moreover, 
        for each $u_j \in \Z[u_1, \dots, u_n] \cong H^\ast(BT)$, 
        $$(\Phi \circ \Lambda^\ast)(u_j)= \sum_{i=1}^m \langle \lambda_i, e_j\rangle x_i.$$
Hence, we conclude that $\text{im}(\pi^\ast \colon H^\ast(BT) \to H^\ast_T(X_\Sigma))=\mathcal{J}$. 
	\end{proof}

\section{Example: Orbifold Hirzebruch varieties}\label{sec_orb.Hirz.Surf}
We finish this paper by illustrating the results of the previous sections by a 
concrete example which is not a weighted projective space.
Consider a primitive vector $(a,b)\in \Z^2$ with $a>0$. Together with $(-1,0),$ $(0,1),$ 
$(0,-1)$, we can make a complete fan $\Sigma$ in $\R^2$ which gives us a 
compact toric variety with two singular points. We denote this toric variety by
$\mathscr{H}_{(a,b)}$. 
See Figure \ref{Fig_Hirzb_and_Orb_Hirzb} for the fan and corresponding 
$\mathcal{R}$-characteristic pair $(Q, \lambda)$. 
When $a=1$, the toric variety is known as a Hirzebruch surface, say $\mathscr{H}_b$. 
In this point of view, let us call $\mathscr{H}_{(a,b)}$ an 
\emph{orbifold Hirzebruch variety}. 
	
Since the collection in \eqref{eq_gcd_collection} becomes 
$\{|G_Q(v)| \mid v\in V(Q)\}=\{1,1,a,a\}$ when $B_1=Q$, its gcd is 1. 
Moreover, in any retraction 
sequence, $B_2$ is  given by a union of edges, which 
guarantees that $(Q, \lambda)$ satisfies the 
assumption of Theorem \ref{thm_gcd_theorem}, 
see Example \ref{ex_orbifold_over_polygon}. 
Moreover, since the underlying polytope is a square, 
the integral betti numbers are given by $\beta^0=\beta^4=1$ and $\beta^2=2$
by Remark \ref{rmk_betti_num=h-vector}. 
\begin{remark}
We may compute the (co)homology groups of low dimensional toric orbifolds 
by the spectral sequence whose $E_1$ page is described by the fan data; 
see \cite{Jor} and \cite{Fis}. More generally, the low dimensional  calculations of 
Kuwata, Masuda and Zeng \cite{KMZ} apply to the category of torus orbifolds. 
\end{remark}

Let $\sigma_{ij}=cone\{\lambda_i, \lambda_j\}$, where $\lambda_1, \dots, \lambda_4$ are 
described in the right hand side of Figure \ref{Fig_Hirzb_and_Orb_Hirzb}. Then, 
the integrality condition of Definition \ref{def_int_cond} is given by the following vectors:
$$
\begin{array}{rrrrrrl}
\vspace{0.1cm}
z^{\sigma_{12}}=&\big(& \frac{1}{a}u_1, &-\frac{b}{a}u_1+u_2,& 0,& 0&\big)\\
\vspace{0.1cm}
z^{\sigma_{14}}=&\big(& \frac{1}{a}u_1,&0,& 0,&  \frac{b}{a}u_1-u_2&\big)\\
\vspace{0.1cm}
z^{\sigma_{23}}=&\big(& 0,& u_2,& -u_1,& 0 &\big)\\
z^{\sigma_{34}}=&\big(&0,&0,&-u_1, &-u_2&\big).
\end{array}
$$
%
Notice that the last two vectors $z^{\sigma_{23}}$ and $z^{\sigma_{34}}$ don't 
contribute the integrality condition,  because their entries have integral coefficients. 

A similar computation to Example \ref{ex_wSR_for_CP^2(1,a,b)} 
shows that the following polynomials are 
elements of degree $2$ in $w\SR[\Sigma]$:
\begin{equation}\label{eq_deg_2_orb_hirz}
ax_1-x_3,~  bx_1+x_2-x_4, ~ ax_1, ~ax_2, ~x_3, ~ax_4
\end{equation}
as are  $\Z$-linear combinations of them. 
The first two elements are actually the linear relations in $\mathcal{J}$, which means 
that they come from the global polynomials in $\PP[\Sigma]$. 
Since rank$H^2(\mathscr{H}_{(a,b)})=2$, we choose two linearly 
independent elements as follows, 
$$w_1:=ax_1\text{ and } w_2:=ax_4.$$ 
Next, degree $4$ elements in $w\mathcal{SR}[\Sigma]$ are
\begin{equation}\label{eq_deg_4_orb_hirz}
a^2x_1^2, ~a^2x_2^2, ~x_3^2, ~a^2x_4^2,~ a^2x_1x_2, 
~a^2x_1x_4, ~x_2x_3\text{ and } ~x_3x_4,
\end{equation}
and their $\Z$-linear combination. 
First four of \eqref{eq_deg_4_orb_hirz} are just the square of degree $2$ elements. 
The remaining four monomials are:
\begin{itemize}
\item $a^2x_1x_2=ax_1ax_2=ax_1a(-bx_1+x_4)=w_1(-bw_1+w_2)$,
\item $a^2x_1x_4=ax_1ax_4=w_1^2$,
\end{itemize}
Notice that the final two monomials $x_2x_3,~x_3x_4$ cannot come from degree $2$ elements. 
Hence, we put 
$$w_3:=x_3x_4.$$ Then, 
$$x_2x_3=(-bx_1+x_4)x_3=x_3x_4=ax_1x_4=aw_3.$$ 
The second equality holds because of the Stanley--Reisner ideal $\mathcal{I}=\langle x_1x_3, x_2x_4\rangle$. 
Finally, the ideal $\mathcal{I}$ and $\mathcal{J}$ determine the multiplicative structures as follows:
\begin{align*}
w_1^2&=(ax_1)^2=(ax_1)(x_3)=0,\\
w_1w_2&=(ax_1)(ax_4)=x_3(ax_4)=aw_3,\\
w_2^2&=(ax_4)(ax_4)=a(bx_1+x_2)(ax_4)=abx_3x_4=abw_3,\\
w_1w_3&=(ax_1)(x_3x_4)=0,\\
w_2w_3&=(ax_4)(x_3x_4)=ax_4x_3(bx_1+x_2)=0,\\
w_3^2&=(x_3x_4)^2=x_3^2x_4^2=(ax_1x_3)(x_4^2)=0.
\end{align*}
Therefore, for the cohomology of  the Hirzebruch variety, we we get the following
 presentation.
\begin{equation}\label{eq_cohom_orb_hirz}
H^\ast(\mathscr{H}_{(a,b)})\cong \Z[w_1, w_2, w_3]/ \langle w_1^2, w_1w_2-aw_3, w_2^2-abw_3, w_1w_3, w_2w_3, w_3^2\rangle, 
\end{equation}
where $\deg w_1=\deg w_2=2$ and $\deg w_3=4$.

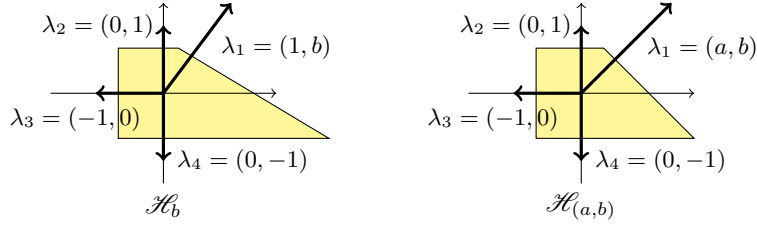
\begin{figure}
	\begin{tikzpicture}[scale=0.3]
		\draw[fill=yellow!50] (3,8)--(3,4)--(37/3, 4)--(17/3,8)--cycle;
		\draw[->] (0,6)--(10,6);
		\draw[->] (5,2)--(5,10);
		\draw[very thick,->] (5,6)--(2,6);
		\node[below] at (1.1,5.8) {\small$\lambda_3=(-1,0)$};
		\draw[very thick,->] (5,6)--(5,9);
		\node[left] at (5,9) {\small$\lambda_2=(0,1)$};
		\draw[very thick,->] (5,6)--(5,3);
		\node[right] at (5.2,3) {\small$\lambda_4=(0,-1)$};
		\draw[very thick,->] (5,6)--(8,10);
		\node[below] at (10,9) {\small$\lambda_1=(1,b)$};
		\node at (5,1) {$\mathscr{H}_b$};
		
		\begin{scope}[xshift=100]
		\draw[fill=yellow!50] (18,8)--(18,4)--(25,4)--(21,8)--cycle;
		\draw[->] (15,6)--(25,6);
		\draw[->] (20,2)--(20,10);
		\draw[very thick,->] (20,6)--(17,6);
		\node[below] at (16.1,5.8) {\small$\lambda_3=(-1,0)$};
		\draw[very thick,->] (20,6)--(20,9);
		\node[left] at (20,9) {\small$\lambda_2=(0,1)$};
		\draw[very thick,->] (20,6)--(20,3);
		\node[right] at (20.2,3) {\small$\lambda_4=(0,-1)$};
		\draw[very thick,->] (20,6)--(24,10);
		\node[right] at (22.5,8) {\small$\lambda_1=(a,b)$};
		\node at (20,1) {$\mathscr{H}_{(a,b)}$};
		\end{scope}
	 \end{tikzpicture}
\caption{A Hirzebruch surface and an orbifold Hirzebruch variety.}
\label{Fig_Hirzb_and_Orb_Hirzb}
\end{figure}

\begin{remark}
The cohomology ring of Hirzebruch surface, by way of comparison, can be computed 
from the result of \cite{Dan}, \cite{Jur} or \cite{DJ}. Indeed it has the following 
presentation 
$$H^\ast(\mathscr{H}_b) \cong \Z[w_1, w_2]/(w_1^2, w^2- b w_1w_2),$$ 
where $\deg w_1=\deg w_2=2$, which means that it is generated by degree $2$ elements. 
However, $H^\ast(\mathscr{H}_{(a,b)})$ has the degree $4$ generator $w_3$ which
 cannot be 
generated by degree $2$ elements, 
i.e., $w_1w_2=a w_3$. Notice that we can recover the presentation of 
$H^\ast(\mathscr{H}_b) $ by replacing $a$ into $1$ in \eqref{eq_cohom_orb_hirz}. 
\end{remark}

%
%
\bibliographystyle{amsalpha}

\end{document}